\documentclass{aptpub}
\usepackage{dsfont}
\usepackage{mathtools}
\usepackage{tikz}
\usepackage{enumitem}
\usepackage{hyperref}
\usepackage{graphicx}
\graphicspath{ {./Files/} }
\usetikzlibrary{tikzmark,calc,positioning,patterns}

\newcommand{\N}{\mathbb{N}}
\newcommand{\M}{\mathbb{M}}
\newcommand{\R}{\mathbb{R}}

\renewcommand*{\S}{\mathbb{S}}
\newcommand{\1}{\mathds{1}}
\newcommand{\Pp}{\mathbb{P}}
\newcommand{\Ascr}{\mathcal{A}}
\newcommand{\Bscr}{\mathcal{B}}
\newcommand{\Cscr}{\mathcal{C}}
\newcommand{\Fscr}{\mathcal{F}}
\newcommand{\E}{\mathbb{E}}
\newcommand{\dint}{\textup{d}}
\newcommand{\vtf}[1]{\mathsf{#1}} 
\newcommand{\betavt}[2]{\vtf{Beta}{(}{#1,#2}{)}}
\newcommand{\dirvt}[1]{\vtf{Dir}{(}{#1}{)}}
\newcommand{\gdirvt}[1]{\vtf{GDir}{(}{#1}{)}}

\newcommand{\dequ}{\overset{\text{d}}{=}}

\newcommand{\enquote}[1]{``#1''}

\newcommand\indep{\protect\mathpalette{\protect\independenT}{\perp}}
\def\independenT#1#2{\mathrel{\rlap{$#1#2$}\mkern2mu{#1#2}}}

\newcommand{\norm}[1]{||#1||}

\DeclarePairedDelimiter\floor{\lfloor}{\rfloor}

\DeclarePairedDelimiter\ceil{\lceil}{\rceil} 

\makeatletter
\def\smallunderbrace#1{\mathop{\vtop{\m@th\ialign{##\crcr
   $\hfil\displaystyle{#1}\hfil$\crcr
   \noalign{\kern3\p@\nointerlineskip}%
   {\tiny\upbracefill}\crcr\noalign{\kern3\p@}}}}\limits}
\makeatother

\newtheorem{notation}[theorem]{Notation}

 \newcommand*\phantomas[3][c]{%
     \ifmmode
         \makebox[\widthof{\(#2\)}][#1]{\(#3\)}%
     \else
         \makebox[\widthof{#2}][#1]{#3}%
     \fi
 }

\definecolor{red}{RGB}{0,0,0}

\authornames{Anja Jan\ss en, Max Ziegenbalg} 
\shorttitle{Multivariate Regular Variation of Preferential Attachment Models} 

\numberwithin{equation}{section}  
\numberwithin{theorem}{section}      
\numberwithin{lemma}{section}
\numberwithin{remark}{section}
\numberwithin{definition}{section}

\begin{document}
\allowdisplaybreaks
\title{Multivariate Regular Variation of Preferential Attachment Models} 

\authorone[Otto-von-Guericke University Magdeburg]{Anja Janssen} 

\addressone{Faculty of Mathematics, IMST, Universitätsplatz 2, 39106 Magdeburg, Germany} 
\emailone{anja.janssen@ovgu.de} 

\authortwo[Otto-von-Guericke University Magdeburg]{Max Ziegenbalg} 
\addresstwo{Faculty of Mathematics, IMST, Universitätsplatz 2, 39106 Magdeburg, Germany}
\emailtwo{max.ziegenbalg@ovgu.de}
\begin{abstract}
We use the framework of multivariate regular variation to analyse the extremal behaviour of preferential attachment models. To this end, we follow a directed linear preferential attachment model for a random, heavy-tailed number of steps in time and treat the incoming edge count of all existing nodes as a random vector of random length. By combining martingale properties, moment bounds and a Breiman type theorem we show that the resulting quantity is multivariate regularly varying, both as a vector of fixed length formed by the edge counts of a finite number of oldest nodes, and also as a vector of random length viewed in sequence space. A P\'{o}lya urn representation allows us to explicitly describe the extremal dependence between the degrees with the help of Dirichlet distributions. As a by-product of our analysis we establish new results for almost sure convergence of the edge counts in sequence space as the number of nodes goes to infinity. 
\end{abstract}



\section{Introduction} 

Preferential attachment random graphs, as initially introduced in \cite{barabasi99} and \cite{bollobas01}, have become prevalent models for the evolution of stochastic networks. The key feature of this class of models is that the current number of edges of a node affects the probability of a new node connecting to it. This underlying ``rich get richer''-dynamic is well-known to lead to a heavy-tailed behavior in the resulting limit distribution of edge counts per node, as the number of nodes goes to infinity. To make this notion of heavy tails more precise, think of a simple preferential attachment model which consists at time $n=0$ of a single node, labelled by $1$, and a single directed edge from this node to itself, i.e.\ a loop. Now, at each following point $n=1,2,\ldots$ in time, we add a new node, labelled by $n+1$, to the graph and attach one directed edge from this new node to itself or another already existing node. This node to connect to is chosen at random, with a probability that depends on the number of incoming edges per node. More precisely, moving from time $n$ to $n+1$ we create a new node that connects to the node with label $i \leq n+2$ with the probability
\begin{equation}\label{Eq:prefattach}  p_i(n):=\frac{\deg_i^{\mbox{\scriptsize{in}}}(n)+\beta}{\sum_{k=1}^{n+2}(\deg_k^{\mbox{\scriptsize{in}}}(n)+\beta)},
\end{equation}
where $\deg_i^{\mbox{\scriptsize{in}}}(n), i=1, \ldots, n+2$ stands for the in-degree of node $i$ at time $n$, i.e.\ the number of incoming edges, and $\beta > 0$ denotes a so-called offset parameter. In this model, the resulting empirical in-degree distribution at time $n$, given by its probability mass function
\begin{equation}\label{Eq:empdeg} i \mapsto p_n(i):=\frac{\sum_{k=1}^{n+1}\mathds{1}_{\{\deg_k^{\mbox{\scriptsize{in}}}(n)=i\}}}{n+1}, \,\;\; n \in \mathbb{N}, 1 \leq i \leq n+1, 
\end{equation}
can be shown to converge as $n \to \infty$ to a limiting probability mass function $i \mapsto p(i), i \in \mathbb{N},$ such that there exists a constant $c>0$ for which
\begin{equation}\label{Eq:empdeglim} \lim_{i \to \infty} \frac{p(i)}{i^{-(2+\beta)}}=c,\end{equation}
 see \cite{bollobasetal03}. We refer to Chapter 8 in \cite{hofstad16} for a concise overview of preferential attachment models, covering both the motivation and a rigorous treatment of their asymptotic behavior. For related work on limit theorems for degree counts, see \cite{krapivskyetal01, mori07, resnicksamorodnitsky16, wangresnick18}.

\begin{figure}[h!]
 \includegraphics[width = 0.64\textwidth]{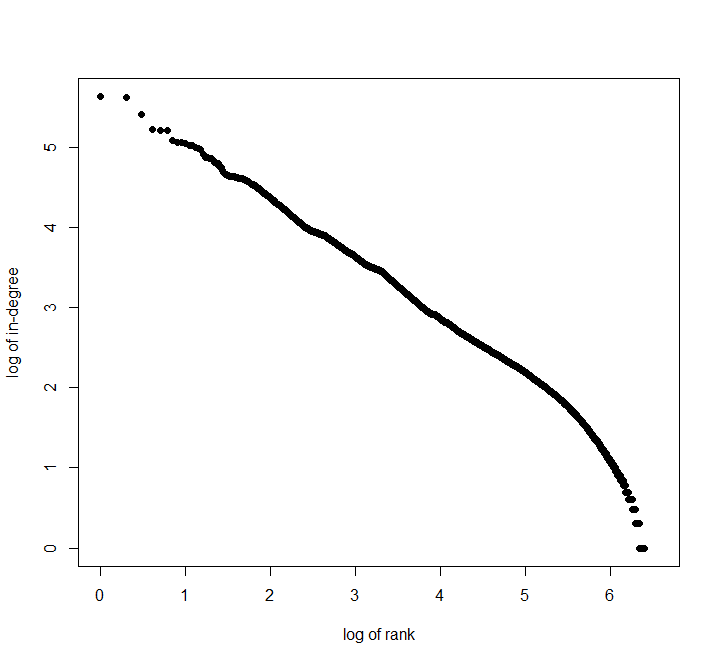}
 \centering
 \caption{\label{fig:zipf2} \textcolor{red}{Zipf plot for in-degress in the network of links between German Wikipedia articles, see \cite{konect13}, \url{http://konect.cc/networks/wikipedia_link_de/}.}}
\end{figure}
\begin{figure}[h!]
 \includegraphics[width = 0.64\textwidth]{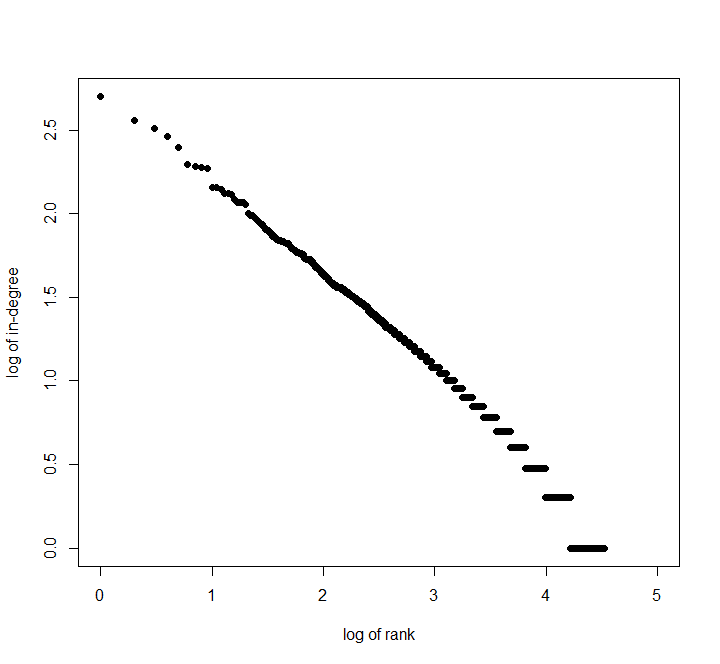}
 \centering
 \caption{\label{fig:zipf} \textcolor{red}{A Zipf plot for in-degress of a simulated preferential attachment model after 100.000 time steps starting from one initial node with offset parameter \(\beta = 1\). For nodes with large degrees (low rank) it shows strong similarities to the real life network from Figure \ref{fig:zipf2}.}}
\end{figure} 
 
Equation~\eqref{Eq:empdeglim} implies that the limiting probability mass function of the in-degrees decays like a power function. In network science, this property of a random graph is typically referred to as ``scale free" and it is often pointed out that a similar power-law behaviour is visible in the empirical degree distributions of real life networks, see \cite{voitalovetal19} for a recent overview on the topic \textcolor{red}{or see the KONECT database \cite{konect13} for an abundance of real life examples.}

\textcolor{red}{One way to illustrate this phenomenon in a data set is by looking at the empirical quantile function of the in-degrees of an observed network. Figures \ref{fig:zipf2} and \ref{fig:zipf} show so-called Zipf plots in the form of the ordered observations of in-degrees on a log-log scale (where rank 1 corresponds to the largest observation), for both a real data set and a simulated realisation of a preferential attachment model. The power law behaviour manifests itself in a clear linear trend in the largest (and for the preferential attachment model: indeed all) observations.} 

\textcolor{red}{However, this kind of asymptotic for the largest observations may not be sufficient to describe the relevant aspects of the extremal behaviour of a preferential attachment model, due to two shortcomings:
\begin{itemize} 
\item Real networks are often evaluated after a random number of nodes have been added to them. For example, if a social network is evaluated at a given date, a random number of participants will have joined it by then. Now, the magnitude of an extremal event, such as the size of the largest degree, primarily depends on the number of participants, and so the distribution of largest degrees depends crucially on the random number of nodes observed.
\item Neither the Zipf plot nor the asymptotic (1.3) provide any information about which nodes will eventually have the largest in-degree. Typically these are the oldest ones, but the exact order of the in-degrees is random and the individual magnitudes show an intricate dependence structure. 
\end{itemize}
As a more concrete example, the asymptotic~\eqref{Eq:empdeglim} is not sufficient to give an estimate for the probability that, say, the sum of the in-degrees of the four oldest nodes exceeds a certain high threshold and that these four in-degrees are in descending order.} 

\textcolor{red}{We will come back to this question in our Example~\ref{Ex:Appl2} below and will more generally provide an asymptotic analysis which allows us to address the two aforementioned shortcomings for a common class of preferential attachment models by using the framework of multivariate regular variation.} To this end, we study a dynamic which is a generalisation of the two models analysed in \cite{pekoz17}, i.e.\ linear directed preferential attachment models in which a new node creates a fixed number of $l \in \mathbb{N}$ outgoing edges, and where we allow for an arbitrary offset parameter $\beta \geq 0$ as in \eqref{Eq:prefattach}. We then follow our network for a \emph{random} number of steps, \textcolor{red}{and assume that this number has a heavy-tailed distribution. It will turn out that both the index of regular variation of this heavy-tailed distribution and the parameters $\beta$ and $l$ of the network dynamic will influence the extremal behavior of the in-degree vector.}
 As tools for our analysis, we extend approaches from \cite{mori05} and \cite{pekoz17} to our more flexible class of models and combine them with a generalisation of a Breiman type result from \cite{wang21} in order to derive the limiting quantities for the asymptotic behaviour of the network. 

 The article is structured as follows: In Section~\ref{Sec:Background} we introduce our model and make a connection to P{\'o}lya urns which allow us to apply martingale methods in order to derive an almost sure limit of mixed powers of finitely many in-degrees as the numbers of nodes goes to infinity (Proposition~\ref{P:martingale_p_factorial}). Next, the concept of regular variation in Banach spaces is introduced in Section~\ref{sec:MRVPARG} along with a Breiman type result (Theorem~\ref{Breiman_generalised}) that allows us to transfer the regular variation of the observed number of nodes to regular variation of the in-degree vector. We first restrict the analysis to a finite stretch of the in-degree vector, i.e.\ we only analyse an arbitrary but fixed number of oldest nodes in the network and derive its regular variation (Theorem~\ref{th:mrvparg} and Corollary~\ref{Cor:DegreeRV}). In order to explicitly describe the extremal dependence structure we deduce the form of the so-called spectral measure. Again by making use of the P{\'o}lya urn representation and its limiting Beta distributions we can show that the spectral measure can be explicitly described with the help of Dirichlet distributions (Theorem~\ref{L:Spectral_distr} and Corollary~\ref{C:Spectral_Measure}). Subsequently, in Section~\ref{sec:Seq_Spaces}, we extend our analysis to sequence space, which allows us for example to keep track of the extremal behaviour of the maximum degree of our network (Corollary~\ref{Cor:MaxRV}). As an auxiliary result we first derive almost sure convergence, as the number of nodes tends to infinity, of the in-degree sequence in sequence space (Proposition~\ref{P:BreiCond}). Finally, we arrive at a regular variation result in sequence space (Theorem~\ref{Th:RVsequence}). 
 Longer proofs and auxiliary results have been deferred to the appendix~\ref{sec:appendix}.
\section{Background on Preferential Attachment Models}\label{Sec:Background}
\subsection{Model Definitions}
In this section we introduce our random graph model and give an equivalent formulation of it that will turn out to be useful. Our model is a generalisation of the model considered in \cite{pekoz17} in that we allow for an arbitrary offset parameter $\beta \geq 0$ in the linear preferential attachment function. 
\subsubsection{The Preferential Attachment Random Graph}\label{Subsec:PAM}
Our object of interest is a graph \(G(n),\,n\in\N_0,\) randomly evolving in a discrete time setting and growing by one vertex in each step from which a fixed amount \(l \in \N\) of edges originate. As per usual a graph is defined as the tuple \(G(n)=(V(n),E(n))\) of its set of vertices and its set of edges. The number of vertices at time \(n\) will be called \(N(n)\) \textcolor{red}{\(=N(0)+n\)}. We will simply number these vertices in their order of appearance, so \(V(n)=\{1,\dots,N(0),N(0)+1,\dots,N(0)+n\}\), where the concrete numbering of the initial \(N(0)\) vertices is of no interest. For the edges, we assume them to be directed and from now on we will only consider the in-degrees \(\deg_k^{\mbox{\scriptsize{in}}}(n), k \in N(n),\) of the vertex $k$ at time \(n\) (as the out-degree is not random). If vertex $k$ does not yet exist at time $n$, we set $\deg_k^{\mbox{\scriptsize{in}}}(n)=0$. Finally, we want to introduce weights \(D_k(n)\) which will determine the probability of a vertex $k$ to get a new edge attached to it at time $n$. We deal with \emph{linear} preferential attachment models and so let \(f:\N_0\to\R_{\geq 0}\) be an affine function with \(f(x)=\alpha x+\beta,\, \alpha,\beta\geq 0\) and apply it to the (random) in-degrees: \(D_k(n):=f(\deg_k^{\mbox{\scriptsize{in}}}(n))\). As the aforementioned probability shall be proportional to $D_k(n)$, without loss of generality we can  restrict ourselves to weight functions \(f\) with \(\alpha=1\), i.e., set
\( D_k(n)=\deg_k^{\mbox{\scriptsize{in}}}(n)+\beta, k \in N(n).\)
By choosing \(\beta=l\) the weight \(D_k(n),\,k>N(0)\) represents the total degree (in-degree \(+\) out-degree) of vertex \(k\).\\
Next, we want to describe the random growth of the graph in more detail. \textcolor{red}{At time \(0\) we start with an arbitrary but deterministic initial graph \(G(0)\) which is finite, meaning that \(|V(0)|+|E(0)|<\infty\).} When transitioning from time \(n\) to \(n+1\) a new vertex is added to the graph establishing a fixed number \(l\in\N\) of edges to already existing vertices in two possible ways. The corresponding figures illustrate as an example how node 4 is added to an existing graph consisting of nodes 1, 2, 3 and edges (solid arrows) in each of the two models. Dashed-dotted arrows show all possibilities for creating a new edge with the corresponding probabilities. We consider two different models.
\begin{flushleft}
\begin{minipage}{0.65\textwidth}
\textbf{Model 0} (no loops):
The first model does not allow for loops, i.e.\ an edge from a vertex to itself. We sequentially add the \(l\) edges, immediately updating the weights of the corresponding vertex. For an \(i\in\{1,\dots,l\}\) let \(K_1,\dots, K_{i-1}\) be the vertices chosen for the first \(i-1\) edges. The probability to attach the \(i\)th edge to vertex \(k\in\{1,\dots,N(n)\}\) then is
\[\frac{D_k(n)+\sum_{j=1}^{i-1}\1_{\{K_j=k\}}}{(i-1)+\sum_{i=1}^{N(n)}D_i(n)},\]
which is the proportion of the weight of \(k\) to the total sum of weights at this particular moment. Note that this formula does not consider the vertex \(N(n)+1\) we are about to add.
\end{minipage}
\begin{minipage}{0.2\textwidth}
\begin{tikzpicture}[
roundnode/.style={circle, draw=black, fill=black!5, thick, minimum size=7mm},
roundnodenew/.style={circle, draw=black!30, fill=black!5, thick, minimum size=7mm},
]
\node[roundnode] (1st) at (-1.5,1.5) {1};
\node[roundnode] (2nd) at (0,0) {2};
\node[roundnode] (3rd) at (1.5,-1.5) {3};
\node[roundnodenew] (4th) at (2,2) {4};

\draw[->, thick] (2nd) -- (1st);
\draw[->, thick] (2nd) .. controls (1,0) and (1.5,-0.5) ..(3rd);
\draw[->, thick] (3rd) -- (2nd);
\draw[->, thick] (3rd) .. controls (-0.5,-1.5) and (-1.5,-0.5) ..(1st);
\draw[->, dashdotted, thick]  (4th) -- (1st) node[midway, anchor=south]{\(\frac{2+\beta}{4+3\beta}\)};
\draw[->, dashdotted, thick]  (4th) -- (2nd) node[midway, anchor=east]{\(\frac{1+\beta}{4+3\beta}\)};
\draw[->, dashdotted, thick]  (4th) -- (3rd) node[midway, anchor=west]{\(\frac{1+\beta}{4+3\beta}\)};
\end{tikzpicture}
\end{minipage}

\end{flushleft}

\begin{flushleft}
\begin{minipage}{0.2\textwidth}
\begin{tikzpicture}[
roundnode/.style={circle, draw=black, fill=black!5, thick, minimum size=7mm},
roundnodenew/.style={circle, draw=black!30, fill=black!5, thick, minimum size=7mm},
]
\node[roundnode] (1st) at (-1.5,1.5) {1};
\node[roundnode] (2nd) at (0,0) {2};
\node[roundnode] (3rd) at (1.5,-1.5) {3};
\node[roundnode] (4th) at (2,2) {4};

\draw[->, thick] (2nd) -- (1st);
\draw[->, thick] (2nd) .. controls (1,0) and (1.5,-0.5) ..(3rd);
\draw[->, thick] (3rd) -- (2nd);
\draw[->, thick] (3rd) .. controls (-0.5,-1.5) and (-1.5,-0.5) ..(1st);
\draw[->, dashdotted, thick]  (4th) -- (1st) node[midway, anchor=south]{\(\frac{2+\beta}{4+4\beta}\)};
\draw[->, dashdotted, thick]  (4th) -- (2nd) node[midway, anchor=east]{\(\frac{1+\beta}{4+4\beta}\)};
\draw[->, dashdotted, thick]  (4th) -- (3rd) node[midway, anchor=west]{\(\frac{1+\beta}{4+4\beta}\)};
\draw[->, dashdotted, thick]  (4th) .. controls (2.2,3.4) and (3.4,2.2) .. (4th) node[midway, anchor=west]{\(\frac{\beta}{4+4\beta}\)};
\end{tikzpicture}
\end{minipage}
\hfill
\begin{minipage}{0.55\textwidth}
\textbf{Model 1} (with loops):
The second model allows for loops by a small modification of the above probability formula. To this end, we set the weight of the new vertex \(D_{N(n)+1}(n)=\beta\). Then, with the same notation as in Model 0, we allow the vertex \(k\) to range over \(\{1,\dots, N(n)+1\}\) and use 
\[\frac{D_k(n)+\sum_{j=1}^{i-1}\1_{\{K_j=k\}}}{(i-1)+\sum_{i=1}^{N(n)+1}D_i(n)}\]
as probability for the \(i\)th edge to attach to it.
\end{minipage}
\end{flushleft}

Thus, the difference between the two models is simply that in Model \(1\) the new vertex is added to the graph before its edges are attached to already existing ones, whereas in Model \(0\) it is added afterwards. 

\subsubsection{An Infinite-Colour Urn Model}\label{sec:InCoUrn}

In order to analyse the limiting behaviour of preferential attachment models, it has proven to be beneficial to exploit a close relationship to generalised P\'{o}lya urns that allows to transfer methods and results from the latter setting to the former. \textcolor{red}{Both P\'{o}lya urns and preferential attachment models are random processes with reinforcement, see \cite{pemantle07} for an overview on the topic. Starting with \cite{berger05}, P\'{o}lya urns have been fruitfully used to study the asymptotic behavior of preferential attachment models, see also \cite{berger14,collevecchio13,garavaglia19, pekoz17, janson21} and \cite[Part II, Section 5]{vanderHofstad24} for related approaches. Our approach below is based on an almost sure representation of one process by the other, given in Lemma~\ref{GraphUrnDuality}, which allows us to derive almost sure convergence statements, thereby complementing and extending results from \cite{pekoz17} and \cite{mori05}.}

\textbf{Urn Model}: 
The idea is that each colour \(k\) in an urn stands for a vertex of the random graph and the number of balls of that colour is the corresponding weight \(D_k\). Since the weights can be any positive real number,  we will also allow our urn to contain non-natural numbers of balls. So let us consider an urn that at time \(0\) contains an arbitrary (natural) number of balls of \(s\) different colours (representing the edges in the graph) plus \(\beta\) balls of each of those colours (the weight functions offset). Then, in the time step \(n\mapsto n+1\), we randomly draw a ball from the urn, where the probability of a certain colour to be drawn, just like in a regular urn, is the ratio of the number of balls of that colour to the total number of balls. Afterwards we return the ball together with an additional one of the same colour. Furthermore, if \(n\) is a multiple of \(l\), we add \(\beta\) more balls of the new colour \(s+\frac{n}{l}\), a mechanic we will refer to as immigration. So one can view $l$ time steps in the urn model as one step for the graph. We will denote the number of balls of colour \(k\) in the urn at time \(n\) with \(C_k(n)\).

We obtain the following lemma describing the aforementioned connection.

\begin{lemma}\label{GraphUrnDuality}
Let \(j \in \{0,1\}\), \(l \in \mathbb{N}\), \(\beta> 0\) and in addition \(N(0)\in \mathbb{N}\) and \((D_1(0),\) \(\dots, D_{N(0)}(0)) \in \mathbb{R}_{\geq 0}^{N(0)}\). Then, there exists a probability space $(\Omega, \Ascr, P)$ which accommodates two stochastic processes \((C_i(n))_{i,n \in \mathbb{N}}\) and \((D_i(n))_{i,n \in \mathbb{N}}\), such that
\begin{itemize}
    \item $(D_i(n))_{i,n \in \mathbb{N}}$ has the same distribution as the random graph Model $j$ described above (where $D_i(n)$ denotes the weight of vertex $i$ after $n$ new vertices have been added) with parameters $l$ and $\beta$ and starting configuration given by $N(0)$ vertices with weights $D_k(0), k=1, \ldots, N(0)$,
	\item $(C_i(n))_{i,n \in \mathbb{N}}$ has the same distribution as the urn model described above (where $C_i(n)$ denotes the number of balls of colour $i$ after $n$ draws have been performed) with parameters $l$ and $\beta$ and starting configuration given by $s:=N(0)+j$ colours, with $D_k(0)$ balls of colour $k=1, \ldots, N(0)$ and, if $j=1$, $\beta$ balls of colour $N(0)+1$, 
\end{itemize}
such that $$ (C_1(n\cdot l),\dots,C_r(n\cdot l))=(D_1(n),\dots,D_r(n)) $$
for any $n \in \N, r<s+n$.
\end{lemma}
\begin{proof}
Both the graph and the urn model for the first $s+n-1$ colours / vertices observed after $n$ steps in the respective model (where one step here means adding one ball of a certain colour / one edge to a certain node) form a Markov chain and one easily checks that the transition probabilities for one step coincide. This ensures the existence of the underlying probability space as stated, see \cite[Section 1.1]{kifer86}.
\end{proof}

\begin{remark}\label{rem:UrnGraph}
The urn model provides several advantages for analysing the random graph in our context. First, it performs each step of adding an edge to the graph separately, i.e., in one draw from the urn, and second, it ignores additional structure in the graph which we are not interested in, such as the information which vertex is hanging on the other side of an edge. Another benefit of the urn model is the elimination of the need to distinguish between versions with and without loops when considering transition probabilities. We simply shift the starting amount of colours in the urn by \(j\), which corresponds either to adding the vertex before attaching its edges, or after. However, when we add the color beforehand, it represents a vertex that should not yet exist, which is why we exclude the case \(r=s+n\) in Lemma~\ref{GraphUrnDuality}. In subsequent proofs, we will frequently alternate between the urn and graph models as permitted by this lemma.
\end{remark}

A first observation we can make about the asymptotic behaviour of the individual weights/ball counts is that they will all tend to infinity almost surely:
\begin{lemma}\label{L:InftyDegree}
Assume the graph / urn model of Lemma \ref{GraphUrnDuality}. Then, as \(n \to \infty\), for arbitrary \(k\) the weight \(D_k(n)\) of vertex \(k\) / ball count \(C_k(n)\) of colour \(k\) diverges to infinity almost surely.
\end{lemma}
\begin{proof}
We take the perspective of the urn model. Consider an arbitrary time \(n=m\cdot l,\,m\in\N\) at which $C_k(n)>0$, where such an $n$ exists since we assumed $\beta>0$. Let \(c\) denote the number of balls of colours other than \(k\) at this time. Then the probability that no further ball of colour \(k\) will be selected from here on is 
\[\prod_{i=0}^\infty\Bigl(1-\frac{C_k(n)}{c+i+\beta\floor{\frac{i}{l}}}\Bigr).\]
The statement follows if we can show that this probability is equal to \(0\) or equivalently that the series
\[-\sum_{i=0}^\infty\ln\Bigl(1-\frac{C_k(n)}{c+i+\beta\floor{\frac{i}{l}}}\Bigr)\]
diverges to infinity. Since \(\ln(1-x)\sim -x\) as \(x\to 0\) said series converges if and only if
\[\sum_{i=0}^\infty\frac{C_k(n)}{c+i+\beta\floor{\frac{i}{l}}}\]
does so. As this expression can be bounded from below by the divergent harmonic series, the result follows.
\end{proof}

\subsection{The Almost Sure Limit and its Moments}\label{Subsec:ASLimit}
By Lemma~\ref{L:InftyDegree} all weights $D_k(n)$ diverge to infinity as time progresses, raising the question about the asymptotic behavior of 
 \[D^r(n):=(D_1(n),\dots,D_r(n)),\]
for an arbitrary but fixed length $r \in \mathbb{N}$ after suitable rescaling when $n \to \infty$ . The following result shows the almost sure convergence and characterises the limiting vector in terms of its mixed moments.
\begin{prop}\label{P:martingale_p_factorial}
 Assume the graph / urn model of Lemma~\ref{GraphUrnDuality}. Throughout this Proposition the times \(n\) or \(m\) are assumed to be greater than or equal to \((r-s)\vee 0\) (in the random graph setting) or  \(l\) times this number (in the urn setting). Define the generalised binomial coefficient 
\[\binom{x}{y}:=\frac{\Gamma(x+1)}{\Gamma(y+1)\Gamma(x-y+1)}\quad \forall x,y \in \mathbb{R}\]
and let \(k_1,\dots,k_r\in \mathbb{R}_{\geq 0}\), \(k:=\sum_{i=1}^r k_i\).
\begin{enumerate}[label=\arabic*)]
\item\label{P:martingale_p_factorial:martingale} There exists a normalising sequence \(c(n,k)\sim n^{k\cdot l/(l+\beta)}\) such that, with respect to the natural filtration \(\Fscr_n:=\sigma(C_k(m),m\leq n, k\leq s+\floor{\frac{n}{l}})\), the process
\begin{align}
\Bigl(\frac{1}{c(n,k)}\prod_{i=1}^r \binom{D_i(n)+k_i-1}{k_i},\Fscr_{n\cdot l}\Bigr)_n\label{term:process_p_factorial}
\end{align}
is a positive  martingale. 
\item\label{P:martingale_p_factorial:asConv} In particular, the above process is almost surely convergent with limit
\begin{align}
\prod_{i=1}^r\frac{\zeta_i^{k_i}}{\Gamma(k_i+1)}\in L^1(\Pp),\label{term:limit_p_factorial}
\end{align}
where \(\zeta_i=\lim_{n\to\infty}n^{-l/(l+\beta)}D_i(n)\).
\item
The limit in 2) closes the martingale to the right, i.e.
\[\E\Bigl(\prod_{i=1}^r\frac{\zeta_i^{k_i}}{\Gamma(k_i+1)}\Big|\Fscr_{n\cdot l}\Bigr)=\frac{1}{c(n,k)}\prod_{i=1}^r \binom{D_i(n)+k_i-1}{k_i} \quad \text{for every }n\in\N.\]
Most notably this means (with \(c(n,0)=1\))
\begin{align}
\E\Bigl(\prod_{i=1}^r\zeta_i^{k_i}\Bigr)=\prod_{i=1}^{r}\Biggl[&\Gamma(k_i+1)\cdot\frac{c((i-s)\vee 0,k_1+\dots+k_{i-1})}{c((i-s)\vee 0,k_1+\dots+k_i)}\nonumber\\
&\cdot\begin{cases}\binom{D_i(0)+k_i-1}{k_i}\quad &\text{if }i\leq N(0)\\ \binom{\beta+k_i-1}{k_i}\quad &\text{if }i> N(0)\end{cases}\Biggr]\label{eq:expec_zetap}
\end{align}
\end{enumerate}
\end{prop}
The proof of this proposition is deferred to Section \ref{Sec:MartingaleProof}.

\section{Regular Variation of the Finite-Dimensional Weight Vector}\label{sec:MRVPARG}
The starting point for the asymptotic analysis in Section~\ref{Subsec:ASLimit} was to let the time index $n$ go to infinity and derive properties of the almost sure limit of the weight vector after proper, time-dependent standardization. In this section, we will take a different approach by focusing on the random vector
   \begin{equation}\label{Eq:D(N):Def} D^r(N):=(D_1(N),\dots,D_r(N)),\end{equation}
of the weights of the oldest $r$ nodes \textcolor{red}{(including the ones already existing at time 0)} evaluated at a \emph{random} time $N$ (where we set $D_k(N)=0$ if node $1 \leq k \leq r$ does not yet exist at time $N$) and how the extremes of this random vector can be described. We can think of $N$ being the random number of steps that our network has already gone through at time of observation. The extremal behavior of $D^r(N)$ will thus both depend on the value of $N$ and the network dynamics. In order to further analyse the behavior, we first introduce a suitable framework.
\subsection{Background on Multivariate Regular Variation}\label{subsec:BackgroundMRV}
When dealing with extremes of random vectors $X$ we typically deem an outcome extreme if its norm $\norm{X}$ exceeds a high threshold and we seek for a stabilising behavior after suitable normalisation if the threshold goes to infinity. This leads to the definition of multivariate regular variation, where the concept can be formalised in several slightly different ways. While the original definition was based on vague convergence, see \cite{deHaan79,resnick87}, we will employ the more flexible concept of $\mathbb{M}$-convergence, see \cite{Hult06,kulik20,lindskog14}, which will allow us to extend our results to the sequence space in Section~\ref{sec:Seq_Spaces}. 
\begin{definition}\label{D:MRV}
Let $(B,\norm{\cdot})$ be a separable Banach space and set \(B^*:=B\setminus \{0\}\), write \(\Bscr(B^*)\) for the Borel $\sigma$-algebra on $B^\ast$ and let \(\M(B^*)\) denote the space of measures on \(\Bscr(B^*)\) which are finite on sets \emph{bounded away from $0$}, i.e.\ sets $A$ such that $\inf_{a \in A}\norm{a}>0$.  
Write furthermore \(\Cscr(B^\ast)\) for the set of all non-negative, bounded and continuous functions \(f\) on \(B^*\) for which there exists an \(\epsilon>0\) such that \(f\) vanishes on \(B_\epsilon(0):=\{b \in B: \|b\|<\epsilon\}\).

We say a random variable \(X\) with values in \(B\) is \emph{multivariate regularly varying in $B^\ast$} if there exists a non-trivial measure \(\mu \in \M(B^*)\), also called the \emph{limit measure}, such that
\begin{equation}\label{Eq:Def:MRV}\frac{\Pp(X/t\in \cdot)}{\Pp(\norm{X}>t)}\overset{t\to\infty}{\to}\mu(\cdot) \text{ in }\mathbb{M}(B^*),\end{equation}
 where 
 \begin{equation}\label{Eq:M:Conv} \mu_t\to\mu\text{ in }\mathbb{M}(B^*)\,\,\,\,\Leftrightarrow\,\,\,\,\mu_t(f):=\int f\dint \mu_t\to\int f\dint \mu=:\mu(f) \quad\forall f\in\Cscr(B^\ast).\end{equation}
The mode of convergence defined in \eqref{Eq:M:Conv} is called \emph{$\M$-convergence}.
\end{definition}
\begin{remark}\label{Rem:RVofNorm} 
Definition~\ref{D:MRV} implies that for any multivariate regularly varying $X$ on $B \setminus \{0\}$, the random variable $\norm{X}$ is (univariate) regularly varying on $\R_{>0}$ with some index \(\alpha\), meaning that
$$\lim_{t \to \infty}\frac{\Pp(\norm{X}>\lambda t)}{\Pp(\norm{X}>t)}=\lambda^{-\alpha} \;\;\; \forall\, \lambda>0.$$  
\end{remark}

For the remainder of this section \ref{sec:MRVPARG}, we will focus on the special case \(B = \mathbb{R}^n\) which also justifies the naming of \emph{multivariate} regular variation. However, in preparation for section \ref{sec:Seq_Spaces}, where we will consider the weight sequence as a whole, we will need to use the more general concept of a separable Banach space. This means that the choice of the norm matters in general (how much so, we will see in section \ref{sec:Seq_Spaces}), while in \(\R^n\), at least for qualitative considerations, it does not. 

The limit measure $\mu$ has a certain structure implied by the multiplicative standardisation used in \eqref{Eq:Def:MRV}.

\begin{theorem}[Theorem 3.1 in \cite{Hult06}]
Let \(X\) be a multivariate regularly varying random variable with limit measure \(\mu\). Then there exists an \(\alpha> 0\) such that for all \(\lambda>0\) and \(A\in\Bscr(B^*)\) we have
\[\mu(\lambda A)=\lambda^{-\alpha}\mu(A).\]
\end{theorem}

We call \(\alpha\) the index of regular variation or tail index. To read the \(\alpha\) directly from the limit measure one may find the second of the following alternative characterisations of multivariate regular variation useful.

\begin{theorem}[Corollary 4.4 in \cite{lindskog14}]\label{th:chara_MRV}
Let \(X\) be a random variable in a separable Banach space \(B\). Each of the following statements is equivalent to \(X\) being multivariate regularly varying with index \(\alpha\).
\begin{enumerate}[label=\alph*)]
\item \label{th:chara_MRV_1} There exist a non-trivial measure \(\mu\) and an increasing function \(b(t)\) such that
\begin{align*}
t\Pp(X/b(t)\in \cdot)\overset{t\to\infty}{\to}\mu\text{ in }\mathbb{M}(B^\ast)\,\,\text{ and }\,\,\frac{b(\lambda t)}{b(t)}\overset{t\to\infty}{\to}\lambda^{\frac{1}{\alpha}}\,\,\forall\lambda>0.
\end{align*}
\item \label{th:chara_MRV_2} There exists a probability measure \(S\) on the unit sphere \(\S_{\norm{\cdot}}:=\{x \in B: \norm{x}=1\}\), also called the spectral measure, such that
\[\frac{\Pp(X/t\in \cdot)}{\Pp(\norm{X}>t)}\overset{t\to\infty}{\to}(\nu_\alpha\otimes S)\circ h^{-1}\text{ in }\mathbb{M}(B^\ast)\]
where \(\nu_\alpha\) is a measure on \(\R_{>0}\) determined by its values on the right-unbounded intervals \(\nu_\alpha(t,\infty):=t^{-\alpha}\) and \(h:\R_{>0}\times\S_{\norm{\cdot}}\to B\) with \(h(x,y)=x\cdot y\) is the polar coordinate transformation. 
\end{enumerate}
\end{theorem}

Note that despite the suggesting notation here, limit measures coming from different normalising functions \(b(t)\) may differ up to a multiplicative factor. To obtain the limit measure from~\eqref{Eq:Def:MRV} in Definition \ref{D:MRV} one may choose the normalising function \(b(t):=F^\leftarrow_{\norm{X}}(1-t^{-1})\), where \(F^{\leftarrow}_{\norm{X}}(u):=\inf\{x \in \R: \Pp(\norm{X} \leq x)\geq u\}\) stands for the generalised inverse of the distribution function of \(\norm{X}\).

\subsection{Main Result}
In this chapter we will start with our analysis of the asymptotic behaviour of the weight vector in the framework of multivariate regular variation. To this end, consider first for fixed \(n,r\in\N\) the random vector $D^r(n)=(D_1(n),\dots,D_r(n))$,
where we set \(D_k(n)=0\) if \(k>N(0)+n\). One immediately finds a deterministic upper bound for the \(\ell_1\)-norm of the vector in 
\[\norm{D^r(n)}_1\leq\norm{D^r(0)}_1+n\cdot(l+\beta) ,\]
which in view of Remark~\ref{Rem:RVofNorm} rules out multivariate regular variation of $D^r(n)$. We thus need the number of nodes in the system to be random in order to witness heavy-tailed behaviour of the weight vector. Viewing the number of nodes or, equivalently, the number of completed steps in the evolution of a network as random is also reasonable from a modelling perspective: We typically do not observe a network after an a priori known number of steps, in particular if new nodes are added according to some random arrival process instead of regularly over time. The results in Section~\ref{Subsec:ASLimit} imply that the number of nodes is the driving factor behind an extremal behaviour of node weights. Our approach is thus to assume a heavy-tailed behaviour of $N$, which we choose to be (univariate) regularly varying, \textcolor{red}{and interpret very large real life networks as the extremal realisations of \(N\)}.
 Many typical regularly varying distributions are continuous, such as the Pareto or Student-$t$ distributions. However, if we round any regularly varying random variable \(Y\) to \(\lfloor Y \rfloor\), the result remains regularly varying, leading to a wide variety of possible distributions for \(N \in \N_0\). 
 
 Now, an often observed principle in extreme value theory is derived from Breiman's Lemma and applies in settings in which, roughly speaking, a system consists of a heavy-tailed component and a light-tailed one. Then, the heavy-tailed component will typically drive the extremal behaviour of the system and determine its index of regular variation, while the light-tailed one may well have an impact on the extremal dependence, in terms of the form of the spectral measure. It turns out that in our case a similar Breiman result, more precisely an extension of Theorem 3 in \cite{wang21} adapted to Banach space-valued processes, applies as well. 
\begin{theorem}\label{Breiman_generalised}
Let \((X_t)_{t\in\R_{\geq 0}}\) be an at least one-sidedly continuous stochastic process with values in a separable Banach space \((B, \norm{\cdot})\) and let \(T\) be a positive random variable, such that the following conditions are satisfied:
\begin{enumerate}[label=\arabic*)]
\item \(T\) is regularly varying with index \(\alpha>0\), i.e. for some scaling function \(b(t)\) we have
	\[t\cdot\Pp\Bigl(\frac{T}{b(t)}>\lambda\Bigr)\overset{t\to\infty}{\to}\lambda^{-\alpha}=\nu_\alpha((\lambda,\infty)) \quad \forall \lambda>0.\]
\item \((X_t)\) and \(T\) are independent, i.e. \(\sigma(T)\indep\sigma(X_t,t\geq 0)\).
\item\label{Brei_generalised:conv} \(X_t\) converges  to some \(X_\infty\in B^\ast\) almost surely as \(t \to \infty\).
\item\label{Brei_generalised:moment} For some \(\alpha'>\alpha\) the following moment condition holds
\[\sup_{t\geq 0}\E\big(\norm{X_t}^{\alpha'}\big)<\infty.\]
\end{enumerate}
Then \(T\cdot X_T\) is multivariate regularly varying in \(B^\ast\) with index \(\alpha\) and
\begin{equation}
t\Pp\Bigl(\frac{T\cdot X_T}{b(t)}\in\cdot\Bigr)\overset{t\to\infty}{\to}[\nu_\alpha \otimes \Pp(X_\infty\in\cdot)]\circ \tilde{h}^{-1}\, \text{ in }\M(B^\ast),\label{eq:Breiman_generalised}
\end{equation}
where \(\tilde{h}:\R_{>0}\times B\to B,\,\tilde{h}(x,y)=x\cdot y\).
\end{theorem}
The proof of this theorem is deferred to Section~\ref{Sec:BreimanProof}.

Using this we can now state and prove our findings:
\begin{theorem}\label{th:mrvparg}
Consider a preferential attachment random graph with the notation as introduced in the last sections. In particular denote with \(l\) and \(\beta\) the number of edges added with each new vertex and the weight functions offset respectively. For fixed \(n,r\in\N\) let \(D^r(n):=(D_1(n),\dots,D_r(n))\) be the corresponding weight vector and let \(N\) be a positive integer-valued random variable. If the following assumptions are satisfied
\begin{enumerate}[label=\arabic*)]
\item\label{th:mrvparg:independence}
\(N\) and \((D^r(n))_{n \in \N}\) are independent and
\item\label{th:mrvparg:rv}
\(N\) is regularly varying with index \(\alpha>0\)
\end{enumerate}
then \(D^r(N)\) is multivariate regularly varying with index \(\alpha\cdot \frac{l+\beta}{l}\) for all \(r \in \N\).
\end{theorem}
\begin{proof}
In order to apply Theorem~\ref{Breiman_generalised} consider the process \((X_t)\) 
\[X_t:=0,\,t<1\qquad X_t:=\frac{1}{t}\cdot D^r(\bigl\lceil t^{\frac{l+\beta}{l}}\bigr\rceil),\, t\geq 1 \]
and the random variable \(T\) with
\[T:=N^{\frac{l}{l+\beta}}\]
Then \(T\cdot X_T = D^r(N)\). Therefore all we need to do is to check the conditions from Theorem~\ref{Breiman_generalised} and determine the index of regular variation.
\begin{enumerate}
\item[1),\,2):]
The independence follows immediately from assumption \ref{th:mrvparg:independence}. As for the regular variation we consider the equivalent formulation of assumption \ref{th:mrvparg:rv}
\[\frac{\Pp(N>x\lambda)}{\Pp(N>x)}\to\lambda^{-\alpha},\quad x\to\infty\]
which implies
\[\frac{\Pp(T>x\lambda)}{\Pp(T>x)}=\frac{\Pp(N>x^{(l+\beta)/l}\lambda^{(l+\beta)/l})}{\Pp(N>x^{(l+\beta)/l})}\to\lambda^{-\alpha\cdot\frac{l+\beta}{l}},\quad x\to\infty\]
i.e. \(T\) is regularly varying with index \(\alpha\cdot\frac{l+\beta}{l}\).
\item[3)]
After the transformation \(t':=t^{l/(l+\beta)}\) we apply Lemma~\ref{P:martingale_p_factorial} to:
\begin{align*}
X_{t'}=\frac{1}{t^{l/(l+\beta)}}D^r(\ceil{t})=\Bigl(\frac{\ceil{t}}{t}\Bigr)^{\frac{l}{l+\beta}}\frac{1}{\ceil{t}^{l/(l+\beta)}}D^r(\ceil{t}),\quad t\geq 1
\end{align*}
and the almost sure convergence follows.

\item[4)]
Because of the definition of \((X_t)\) we will only consider \(t\geq 1\) and we set again \(t':=t^{\frac{l}{l+\beta}}\). For \(q\in\N\) we have by Lemma~\ref{GraphUrnDuality} and \textcolor{red}{the moment bounds from} Lemma~\ref{L:MomentCond} \textcolor{red}{from the Appendix} that
\begin{align*}
\E(\norm{X_{t'}}_1^q)&=\frac{1}{t^{ql/(l+\beta)}}\E([D_1(\ceil{t})+\dots+D_r(\ceil{t})]^q)\\
&=\frac{1}{t^{ql/(l+\beta)}}\E((C_1(l\cdot\ceil{t})+\dots+C_r(l\cdot\ceil{t}))^q)\\
&\leq C\cdot\frac{\ceil{t}^{ql/(l+\beta)}}{t^{ql/(l+\beta)}}\leq C\cdot 2^{\frac{ql}{l+\beta}}
\end{align*}
for some \(C>0\), which completes the verification of the conditions of Theorem~\ref{Breiman_generalised}.
\end{enumerate}
\end{proof}

For the actual in-degree vector without the weight functions offset, we obtain the following corollary.
\begin{cor}\label{Cor:DegreeRV}
Under the same assumptions as in Theorem~\ref{th:mrvparg} the vector \(D^r(N)-\beta\) is multivariate regularly varying with the same index of regular variation and the same limit measure as \(D^r(N)\). 
\end{cor}
\begin{proof}
This is a consequence of Lemma~3.12 from \cite{mikosch06}.
\end{proof}

\subsection{Characterisation of the Spectral Measure}\label{sec:SpecMeas}
Having verified the multivariate regular variation of the vector \(D^r(N)\), the natural next step is to characterise its spectral measure \(S\). Over the next few lemmata we shall once more switch into the urn perspective for this purpose. Because of the close relation to P{\'o}lya urns the following Beta and Beta-related distributions will come up, which we repeat here for convenience and since there exist different parameterisations in the literature. 
\begin{definition}\(\)
\begin{enumerate}[label=\arabic*)]
\item A random variable on \([0,1]\) is said to follow a beta distribution \(\betavt{a}{b}\), with parameters \(a,b>0\) if it has a density function
\[f(x)=\frac{1}{\text{B}(a,b)}x^{a-1}(1-x)^{b-1}\cdot \1_{[0,1]}(x),\]
where \(\text{B}(a,b)=\frac{\Gamma(a)\Gamma(b)}{\Gamma(a+b)}\) is the beta function.
\item 
A random vector with values in the standard ($r$-1)-simplex in \(\R^r_{\geq 0}\), i.e.\ the set \(\Sigma^{r-1}:=\{x\textcolor{red}{=(x_1,\dots,x_r)}\in\R^r_{\geq 0}\,:\,x_1+\dots+x_r=1\}\), is said to follow a Dirichlet distribution \(\dirvt{a_1,\dots,a_r}\) with parameters \(a_1,\dots,a_r>0\) if it has a density function
\begin{align*}
f(x)=\frac{1}{\text{B}(a_1,\dots,a_r)}\prod_{i=1}^{r}x_i^{a_i-1}\cdot \1_{\Sigma^{r-1}}(x),
\end{align*}
where \(\text{B}(a_1,\dots,a_r):=\frac{\Gamma(a_1)\cdots\Gamma(a_r)}{\Gamma(a_1+\dots+a_r)}\) is the multivariate beta function.
\item A random vector with values in the standard ($r$-1)-simplex in \(\R_{\geq 0}^r\) is said to follow a generalised Dirichlet distribution \(\gdirvt{a_1,b_1,\dots,a_{r-1},b_{r-1}}\) with parameters \(a_1,b_1,\dots,a_{r-1},b_{r-1}>0\) if it has a density function:
\begin{align*}
f(x)=\frac{1}{\prod_{i=1}^{r-1}\text{B}(a_i,b_i)}x_r^{b_{r-1}-1}\prod_{i=1}^{r-1}\Bigl(x_i^{a_i-1}\Bigl(\sum_{j=i}^rx_j\Bigr)^{b_{i-1}-(a_i+b_i)}\Bigr)\cdot \1_{\Sigma^{r-1}}(x),
\end{align*}
where \(b_0\) is arbitrary.
\end{enumerate}
\end{definition}

\begin{remark} \label{Rem:ProofSketch_Dir}
The generalised Dirichlet distribution boils down to the standard version if its parameters satisfy \(b_{i-1}=a_i+b_i\). Its density function then depends only on the parameters \(a_1,\dots,a_{r-1}\) and \(b_{r-1}\) which can be viewed as the parameters \(a_1,\dots,a_r\) of the standard Dirichlet distribution. These two distributions are also connected to the beta distribution through a stick breaking experiment, in which a random vector is constructed by breaking independent, beta distributed fractions of a unit length stick for a fixed number of times, and the lengths of the resulting pieces form the components of the random vector. In general, this approach leads to a generalised Dirichlet distribution with parameters equal to those of the beta distributed fractions in the order in which they were broken off (see \cite{mosimann69}). To obtain a standard Dirichlet vector, the beta distributions must satisfy the parameter constraints mentioned above.
\end{remark}

We begin with a well-known result for traditional multi-colour P{\'o}lya urns, i.e. our urn model but without immigration. The proof can be found in \cite{blackwell73}.
\begin{lemma}\label{L:Polya_Dir}
Consider an urn with balls of \(r\in\N\) different colours and starting amounts \(a^r=(a_1,\dots,a_r)\in\R^r_{>0}\). Let \((X(n))_{n\in\N}\) be the sequence of colours drawn in each time step by following a traditional multi-colour P{\'o}lya urn scheme where we add one ball of the drawn colour in each step. Then, for each colour \(i\), the proportions of ball counts
\[\frac{a_i+\sum_{k=1}^n\1_{\{i\}}(X(k))}{\sum_{j=1}^ra_j+n},\quad i=1,\dots,r\]
converge, for $n \to \infty$, almost surely to some random variable \(Y_i\) such that:
\[Y^r:=(Y_1,\dots,Y_r)\dequ\dirvt{a_1,\dots,a_r}.\]
\end{lemma}
\begin{notation}\label{not:Xr}
In order to describe the spectral measure it is convenient to look at the projection of the limiting vector from Proposition~\ref{P:martingale_p_factorial}, 2), i.e.
$$ \zeta^r:=(\zeta_1, \ldots, \zeta_r)=(\lim_{n\to\infty}n^{-l/(l+\beta)}D_i(n))_{1 \leq i \leq r},$$ on the \(\ell_1\)-unit sphere along with some derived proportions. We are specifically interested in two orderings of this vector: forwards
\begin{align*}
S^r&:=(S_1, \ldots, S_r):=(\zeta_1,\dots,\zeta_r)/(\zeta_1+\dots+\zeta_r), \; \mbox{with}\\
B_k&:=\frac{S_k}{S_k+\dots+S_r}=\frac{\zeta_k}{\zeta_k+\dots+\zeta_r},\quad 1\leq k \leq r\\
\intertext{and backwards}
S^r_{\downarrow}&=(S_r,\dots,S_1),\,\mbox{ with } B_k^{\downarrow}:=\frac{S_k}{S_1+\dots+S_k}=\frac{\zeta_k}{\zeta_1+\dots+\zeta_k},\quad 1\leq k \leq r.
\end{align*}
\end{notation}
In order to apply Lemma~\ref{L:Polya_Dir} to our problem one needs to observe that once all colours we want to consider are present in the urn (which is at time \(n_0^r=l(r-s) \vee 0\)), these colours, considered separately from the rest of the urn, will behave just like they would in a traditional multi-colour P{\'o}lya urn. We thus obtain the following mixture distribution for \(S^r\).

\begin{theorem}\label{L:Spectral_distr}
Let \(r\in\N\), \(C^r(n):=(C_1(n),\dots,C_r(n))\), $n \in \mathbb{N},$ be the vector of the first $r$ colours in the urn model after $n$ steps with $s$ initial colours from Lemma~\ref{GraphUrnDuality} and let \(S^r\) be as in Notation \ref{not:Xr} the $\ell_1$-projection of the limiting vector in the corresponding random graph model. 
 \begin{enumerate}[label=\arabic*)]
\item \label{L:Spectral_distr_1} The limiting vector \(S^r\) follows a mixture of Dirichlet distributions, i.e.
\[S^r\dequ\,\sum_{\mathclap{c\in\R^r_{\geq 0}}}\,\Pp(C^r(n_0^r)=c)\cdot\dirvt{c_1,\dots,c_r}.\]
\item \label{L:Spectral_distr_2} For the vector obtained by inverting the order of components in \(S^r\) we additionally have
\[S^r_{\downarrow}\dequ \gdirvt{a_r,b_r,\dots,a_2,b_2}\]
where 
$$(a_k,b_k)=\begin{cases}(C_k(0), \sum_{i=1}^{k-1}C_i(0)) \quad&\text{if } 2 \leq k\leq s\\ (\beta,\sum_{i=1}^{s}C_i(0)+(l+\beta)(k-1-s)+l)\quad &\text{if } k>s\end{cases}, $$
\end{enumerate}
\end{theorem}
The proof relies on the following lemma.
\begin{lemma}\label{L:Hilfs_Beta} Let \(B_k,\,B_k^\downarrow\), \(1\leq k\leq r\) be as in Notation \ref{not:Xr}. We have
\begin{enumerate}[label=\arabic*)]
\item \label{L:Hilfs_Beta:1} 
\(B_1^\downarrow\equiv 1\)\\
\(\displaystyle B_k^\downarrow\dequ \begin{cases}\betavt{C_k(0)}{\sum_{i=1}^{k-1}C_i(0)} \quad&\text{if } 2 \leq k\leq s\\ \betavt{\beta}{\sum_{i=1}^{s}C_i(0)+(l+\beta)(k-1-s)+l}\quad &\text{if } k>s\end{cases}\)
\item \label{L:Hilfs_Beta:2} The \(B_1^\downarrow,\dots,B_r^\downarrow\) are independent.
\item \label{L:Hilfs_Beta:3} Both \((B_1,\dots,B_r)\) and \((B_1^\downarrow,\dots,B_r^\downarrow)\) are independent of \(\zeta_1+\dots+\zeta_r\).
\end{enumerate}
\end{lemma}
The proof of this lemma is deferred to section \ref{Sec:BetaIndProof}.

\begin{proof}[Proof of Theorem~\ref{L:Spectral_distr}]
\begin{enumerate}[label=\arabic*)]
\item We start by conditioning the urn model on its composition at time \(n_0^r\), which is how the mixture distributions in the statement arise. Given that \(C^r(n_0^r)=(c_1,\dots,c_r)\in\R^r_{\geq 0}\), let \(t_i\) be the random time where for the \(i\)th time a ball from one of the colours \(1\) to \(r\) is chosen. At these times the vector of ratios of balls of each colour to the total number of balls with colours \(1\) to \(r\) behaves just like in a traditional multi-colour P{\'o}lya urn and therefore by Lemma \ref{L:Polya_Dir} converges almost surely to a Dirichlet distribution with \(c_1,\dots,c_r\) as parameters. At the same time, however, it converges (conditionally on \(C^r(n_0^r)=(c_1,\dots,c_r)\)) to \(S^r\), and so the statement follows.
\item We write 
$$ S_{\downarrow}^r=\left(B_r^{\downarrow},(1-B_r^{\downarrow})B_{r-1}^{\downarrow}, \ldots, \left(\prod_{i=2}^r (1-B_{i}^{\downarrow})\right)B_{1}^{\downarrow}\right)$$
and since by Lemma~\ref{L:Hilfs_Beta} the $B_{i}^{\downarrow}$'s are independent and have beta distributions, we can apply Remark \ref{Rem:ProofSketch_Dir} to arrive at the generalised Dirichlet distribution with the given parameters.
\end{enumerate}
\end{proof}

With these tools, we are now equipped to characterize the spectral measure.
\begin{cor}\label{C:Spectral_Measure}
The spectral measure of regular variation for the random vectors \((D_1(N),\) \(\dots,D_r(N))\) and \((D_r(N),\dots,D_1(N))\), with respect to the norm \(\norm{\cdot}_1\), coincides with the distribution of \(S^{r}\) and \(S^r_\downarrow\), respectively, from Theorem \ref{L:Spectral_distr}.
\end{cor}
\begin{proof}
We restrict ourselves to the forwards case, i.e. focus on \((D_1(N),\) \(\dots,D_r(N))=D^r(N)\). By Equation \eqref{eq:Breiman_generalised} for the limit measure \(\mu\) provided in Theorem~\ref{Breiman_generalised} and with $X_\infty$ given by $\zeta^r$ we get \[\mu=( \nu_\alpha \otimes \Pp(\zeta^r\in\cdot))\circ \tilde{h}^{-1}\]
with \(\tilde{h}:\R_{>0} \times \R^r_{\geq 0}\to\R^r_{\geq 0},\,h(x,y)=x\cdot y\) and \(\zeta^r=(\zeta_1,\dots,\zeta_r)\). On the other hand, by Theorem~\ref{th:chara_MRV}, b) the corresponding spectral measure $S(A)$, for an $S$-continuity Borel set \(A \subset \mathbb{S}_{\| \cdot \|}\) is given by
\begin{align*} S(A)&=\lim_{t \to \infty} \frac{\Pp\left(\frac{D^r(N)}{\|(D^r(N)\|_1} \in A, \|(D^r(N)\|_1>t\right)}{P(\|(D^r(N)\|_1>t)}\\
&=\frac{\mu(\{x \in \R^r_{\geq 0}: \| x \|_1 >1, x/\|x\|_1 \in A\})}{\mu(\{x \in \R^r_{\geq 0}: \| x \|_1 >1\})}.\end{align*}
Now,
\begin{align*}
& \quad \; \mu(\{x \in \mathbb{R}^r: \| x \|_1 >1, x/\|x\|_1 \in A\} \\
&=\int_{\R^r_{\geq 0}}\int_{\R_{>0}}\1_{\{x\cdot\| y\|_1>1\}}\1_{\{\frac{y}{\| y\|_1}\in A\}}\,\nu_\alpha(\dint x)\,\Pp(\zeta^r\in\dint y)\\
&=\int_{\R^r_{\geq 0}}\| y\|_1^\alpha\1_{\{\frac{y}{\| y\|_1}\in A\}}\,\Pp(\zeta^r\in\dint y)\\
&=\E(\|\zeta^r\|_1^\alpha)\cdot\Pp\Bigl(\frac{\zeta^r}{\|\zeta^r\|_1}\in A\Bigr),
\end{align*}
where the last equality follows from the independence of \(\|\zeta^r\|_1=\zeta_1+\dots+\zeta_r\) and \(\frac{\zeta^r}{\|\zeta^r\|_1}=f(B_1,\dots,B_{r})\) (statement \ref{L:Hilfs_Beta:3} in Lemma \ref{L:Hilfs_Beta}), with \(f=(f_1, \ldots, f_r)\) given by:
\[f_i(B_1,\dots,B_r)=B_i\prod_{j=1}^{i-1}(1-B_j), \;\;\; i=1, \ldots, r.\]
Going back to determining \(S\) we thus get
\[S(A)=\frac{\E(\|\zeta^r\|_1^\alpha)\cdot\Pp(\frac{\zeta^r}{\|\zeta^r\|_1}\in A)}{\E(\|\zeta^r\|_1^\alpha)}=\Pp\Bigl(\frac{(\zeta_1,\dots,\zeta_r)}{\zeta_1+\dots+\zeta_r}\in A\Bigr)=\Pp(S^r\in A)\]
and the statement follows.

\end{proof}

We conclude this section with an example that demonstrates how the previous results can be applied to approximate the probabilities of extremal events, starting with a general approach and followed by a specific numerical example.
\begin{ex}\label{Ex:Appl1}
We can use the limit/spectral measure to approximate conditional probabilities given large exceedances of the weight vector. A simple special case is to condition on \(\{\|D^r(N)\|_1> t\}\) for some large \(t\). Let \(A\) be a set in \(\Bscr(\R^r_{\geq 0}\setminus\{0\})\) bounded away from \(0\) and with \(\mu(\partial A)=0\). Then we have, using Theorem \ref{th:chara_MRV} \ref{th:chara_MRV_2} (and the polar coordinate function $h$ as defined there)
\begin{align*}
\Pp(D^r(N)\in tA\,|\,\|D^r(N)\|_1>t)&=\frac{\Pp(D^r(N)/t\in A\cap\{x:\|x\|_1> 1\})}{\Pp(\|D^r(N)\|_1> t)}\\
&\approx (\nu_\alpha\otimes S)\circ h^{-1}(A\cap\{x:\|x\|_1> 1\}),
\end{align*}
which does not depend on \(t\) anymore. So as long as we know that our \(r\) vertices of interest have received a large amount of edges, we can approximate above probability using the limit/spectral measure. 

Next, let us assume that \(A\) only contains information about the proportions of the \(D_1(N),\dots,D_r(N)\), i.e. that \(A\) is of the form \(\{r\cdot\theta|\,(r,\theta)\in(1,\infty)\times A^*\}\) with \(A^*\in\Bscr(\S_{\|\cdot\|_1})\), then 
\begin{align*}\Pp(D^r(N)\in tA\,|\,\|D^r(N)\|_1\geq t)& \approx (\nu_\alpha\otimes S)\circ h^{-1}(A\cap\{x:\|x\|_1> 1\})\\
&= (\nu_\alpha\otimes S)\circ h^{-1}(\{r \cdot \theta: r>1, \theta \in A^*\})\\
&=S(A^*).
\end{align*}
So, given that the total weight of our vector \(D^r(N)\) is large, the proportions of its components follow roughly a mixture of Dirichlet distributions or a reversed generalised Dirichlet distribution.
\end{ex}
\begin{ex}\label{Ex:Appl2}\textcolor{red}{
Following up on Example \ref{Ex:Appl1}, let us consider a concrete extremal event. We want to determine the probability that the first four vertices are extreme with respect to the \(\norm{\cdot}_1\)-norm, while also being in descending order, i.e., with 
\[A=\{r\cdot\theta|\,(r,\theta)\in(1,\infty)\times A^*\}\;\; \mbox{for}\;\; A^*:=\{(x_1,x_2,x_3,x_4)\in\S_{\|\cdot\|_1}|\,x_1\geq x_2\geq x_3\geq x_4\}\]
we want to approximate
\begin{align*}
\Pp(D^4(N)\in tA,\norm{D^4(N)}_1>t)=\Pp(\norm{D^4(N)}_1>t)\cdot\Pp(D^4(N)\in tA\,|\,\norm{D^4(N)}_1>t).
\end{align*}
The second factor can be approximated by the spectral measure as explained in Example \ref{Ex:Appl1}. For the first one, we get, using Theorem \ref{Breiman_generalised},
\begin{align*}
\frac{\Pp(\norm{D^4(N)}_1>t)}{\Pp(N^{l/(l+\beta)}>t)}=u(t)\cdot\Pp\Bigl(\frac{\norm{D^4(N)}_1}{b(u(t))}>\frac{t}{b(u(t))}\Bigr)\overset{t\to\infty}{\to}\E\bigl(\norm{\zeta^4}_1^{\alpha\cdot\frac{l+\beta}{l}}\bigr),
\end{align*}
where \(u(t):=\Pp(N^{\frac{l}{l+\beta}}>t)^{-1}\) and $b:=u^{\leftarrow}$ is the generalised inverse of the monotone function $u$ with \(b(u(t)) \sim t\), see \cite{dehaan06} B1.9 10. Overall, we obtain for large \(t\)
\begin{align}\label{Eq:ExampApprox}
\Pp(D^4(N)\in A,\norm{D^4(N)}_1>t)\approx \Pp(N^{l/(l+\beta)}>t)\cdot\E\bigl(\norm{\zeta^4}_1^{\alpha\cdot\frac{l+\beta}{l}}\bigr)\cdot S(A^*).
\end{align}
To get an impression of the quality of the approximation, we present the results of a small simulation study. The random value of $N$ in our simulations is given by $\lfloor Y \rfloor$, where $Y$ follows a Pareto($\alpha$) distribution with $\alpha=1$, i.e. $P(N \geq k)=k^{-1}, k \in \mathbb{N}$. Table \ref{tabl:sim} shows values of the above approximation \eqref{Eq:ExampApprox} for different choices of model parameters \(l\) and \(\beta\) and threshold $t$, where $\E\bigl(\norm{\zeta^4}_1^{\alpha\cdot\frac{l+\beta}{l}}\bigr)$ follows from \eqref{eq:expec_zetap} and $S(A^*)$ has been evaluated by numerical integration with respect to the density of the generalized Dirichlet distribution. Since no closed expressions exist for the true probabilities, we replace them by empirical probabilities derived from $10^7$ realisations of the network. 
\begin{table}
\begin{centering}
\begin{tabular}{c|c|c|c|c}
 \(l\) & \(\beta\) & \(t\) & empirical probability & approximation  \\\hline
   1   &     1     &  150  &    \(0.00968\%\)     & \(0.00949\%\) \\\hline
   3   &     1     &  500  &    \(0.06325\%\)      & \(0.06215\%\)  \\\hline
   3   &     3     &  500  &    \(0.01272\%\)      & \(0.01255\%\)
\end{tabular}
\caption{\label{tabl:sim}\textcolor{red}{Comparison of left and right hand side \eqref{Eq:ExampApprox} for several parameter constellations. The left hand side was approximated by empirical probabilities based on $10^7$ realisations of the network.}}
\end{centering}
\end{table}
}
\end{ex}

\section{Generalisation to Sequence Spaces}\label{sec:Seq_Spaces}

\subsection{Motivation and Framework}
Example \ref{Ex:Appl2} has shown that the results developed in the previous chapter allow us to approximate probabilities of extreme events - but with the restriction that those extreme events may only depend on a fixed number $r$ of the first nodes in the graph. For several natural applications this approach is then not sufficient, as we would for example like to approximate probabilities of events involving the maximum degree of all existing nodes. In this section, we will extend our scope to include the asymptotic behavior of the entire weight process rather than just its first \(r\) vertices. We will thus again work in the framework of regular variation as in Section~\ref{subsec:BackgroundMRV}, but with a suitably adjusted Banach space $B$ for sequences. The regular variation of random sequences has previously been studied in \cite{tillier18}. In order to find an appropriate $B$ we start with the space of all finite sequences:
\[c_{00}:=\{x=(x_n)_{n\in\N}\in\R^\N|\exists N\in\N\,\forall k\geq N: x_k=0\}.\]
At any time \(n\) our weight process as well as the urn process can be represented as elements of this space:
\begin{align*}
D(n)&:=(D_1(n),\dots,D_{N(0)}(n),D_{N(1)}(n),\dots,D_{N(n)}(n),0,\dots),\\
C(n)&:=(C_1(n\cdot l),\dots,C_{s}(n\cdot l),C_{s+1}(n\cdot l),\dots,C_{s+n}(n\cdot l),0,\dots).
\end{align*}
 A crucial assumption in Theorem \ref{Breiman_generalised} is the almost sure convergence of the process and so we set $B$ equal to the completion of \(c_{00}\) with respect to a norm \(\|\cdot\|\) on $\R^\N$, i.e.
\[B=c_{\|\cdot\|}:=\bigl\{x\in\R^\N\big|\,\|x\|<\infty\text{ and }\exists \, x_n \in c_{00}, n \in \N:\lim_{n\to\infty}\|x-x_n\|= 0\bigr\}.\] In contrast to the previously studied finite dimensional setting, different norms on sequence spaces are no longer equivalent, which means that in order to study regular variation we have to find a suitable norm \(\|\cdot\|\) for our model, where we restrict ourselves to the $\ell_p$-norms. Then, by construction, $B$ is a separable Banach space which allows us to employ the framework from Section~\ref{subsec:BackgroundMRV} again.

\subsection{The Breiman Conditions in Sequence Space}
Our goal is to establish a result similar to Theorem \ref{th:mrvparg} again by virtue of Breiman's Lemma. To this end, we need to check that Conditions \ref{Brei_generalised:conv} and \ref{Brei_generalised:moment} of Theorem~\ref{Breiman_generalised} are satisfied, which is shown in the following Proposition.
\begin{prop}\label{P:BreiCond} Assume the graph / urn model of Lemma~\ref{GraphUrnDuality} and let 
$$ \zeta:=(\zeta_i)_{i \in \mathbb{N}}=\left(\lim_{n \to \infty} \frac{D_i(n)}{n^{l/(l+\beta)}}\right)_{i \in \N},$$ cf.\ Proposition~\ref{P:martingale_p_factorial}, 2). Furthermore, let \(p\in[1,\infty]\).
\begin{itemize}
    \item If \(p>\frac{l+\beta}{l}\), then
\begin{enumerate}[label=\arabic*)]
\item \label{P:BreiCond:zetainc} \(\zeta\in c_{\norm{\cdot}_p}\) a.s. since both 
\begin{enumerate}[ref=(\alph*)]
\item \label{P:BreiCond:zeta_p-smble} \(\norm{\zeta}_p<\infty\) a.s. and
\item \label{P:BreiCond:p-conv} \(\begin{aligned} \frac{D(n)}{n^{l/(l+\beta)}}\to\zeta \text{ a.s. in } c_{\norm{\cdot}_p} \end{aligned}\) as \(n\to\infty\)
\end{enumerate}
\item \label{P:BreiCond:moments} \(\begin{aligned}\sup_n\E\Bigl(\Bigl\lVert\frac{D(n)}{n^{l/(l+\beta)}}\Bigr\rVert_p^\alpha\Bigr)<\infty\end{aligned}\) for every \(\alpha>0\).
\end{enumerate}
 \item If \(p<\frac{l+\beta}{l}\) then 
\begin{enumerate}[resume, label=\arabic*)]
\item\label{P:BreiCond:no_p-conv} \(\begin{aligned} \frac{D(n)}{n^{l/(l+\beta)}}\not\to\zeta \text{ a.s. in } c_{\norm{\cdot}_p} \end{aligned}\) as \(n\to\infty\).
\end{enumerate}
\end{itemize}
\end{prop}

The proof of this proposition is deferred to Section \ref{Sec:BreiCondProof}.

\begin{remark}
Proposition \ref{P:BreiCond} generalises the convergence results of Theorem 2 in \cite{pekoz17} from weak to almost sure convergence and allows for general $l$ and $\beta$. 
\end{remark}
\subsection{Regular Variation in Sequence Space}
\begin{theorem}\label{Th:RVsequence}
Consider a preferential attachment random graph with the notation as introduced in the last sections. Let \(D(n):=(D_1(n),\) \(D_2(n),\dots)\) be the corresponding weight sequence at time \(n\in\N\) and let \(N\) be a positive integer-valued random variable. Let \(p \in (\frac{l+\beta}{l},\infty]\) and \(\norm{\cdot}\) be a norm such that there exists some \(C>0\) with \(\norm{\cdot}\leq C\norm{\cdot}_p\). If the following conditions are satisfied:
\begin{enumerate}[label=\arabic*)]
\item
\(N\) and \((D(n))_n\) are independent and
\item
\(N\) is regularly varying with index \(\alpha>0\)
\end{enumerate}
then \(D(N)\) is multivariate regularly varying in \(c_{\norm{\cdot}}\) with index \(\alpha\cdot \frac{l+\beta}{l}\) .
\end{theorem}
\begin{proof} Note first that our assumptions guarantee that the convergence results of Proposition \ref{P:BreiCond} hold analogously with the norm $\|\cdot \|_p$ replaced by $\|\cdot \|$. 
Thus, the assumptions of Theorem \ref{Breiman_generalised}) are met and the proof is completely analogous to the proof of Theorem \ref{th:mrvparg}.
\end{proof}
This result finally allows us to derive (univariate) regular variation of the maximum degree of a preferential attachment model.
\begin{cor}\label{Cor:MaxRV}
Under the assumptions of Theorem \ref{Th:RVsequence} the maximum degree $sup_{i \in \mathbb{N}}D_i(N)$ is regularly varying with index $\alpha \cdot \frac{l+\beta}{l}$.
\end{cor}
\begin{proof} This follows from Theorem \ref{Th:RVsequence} applied to $\| \cdot \|_\infty$ and Remark \ref{Rem:RVofNorm}.
\end{proof}
\appendix

\section{Auxiliary Results and Deferred Proofs}\label{sec:appendix}

\subsection{Proof of Proposition \ref{P:martingale_p_factorial} and Uniform Integrability}\label{Sec:MartingaleProof}
For the proof of several lemmata and other auxiliary results we make use of a corollary to Stirling's formula, that can be found as 6.1.46 in \cite{abramowitz84} and reads
\begin{lemma}\label{L:Stirling}
\begin{align*}
\lim_{n\to\infty}n^{b-a}\frac{\Gamma(a+n)}{\Gamma(b+n)}=1\quad\forall a,b\in\R
\end{align*}
\end{lemma}
This allows us to estimate binomial coefficients by their respective powers and vice versa.
\begin{lemma}\label{L:ineq_mart_power_form}
Under the assumptions of Proposition~\ref{P:martingale_p_factorial}, we have
\[\prod_{i=1}^r\frac{D_i^{k_i}(n)}{\Gamma(k_i+1)}\sim \prod_{i=1}^r\binom{D_i(n)+k_i-1}{k_i}\quad \text{ a.s.}\]
and there exist constants \(C_1,C_2>0\) independent of \(n\) such that almost surely
\begin{align}
C_1\cdot\prod_{i=1}^r\binom{D_i(n)+k_i-1}{k_i}\leq \prod_{i=1}^rD_i^{k_i}(n) \leq C_2\cdot\prod_{i=1}^r\binom{D_i(n)+k_i-1}{k_i}\label{eq:mart_power_form}
\end{align}
\end{lemma}
\begin{proof}
Remember from Lemma~\ref{L:InftyDegree} that almost surely \(D_i(n)\to\infty\) as \(n\to\infty\). Thus, by Lemma \ref{L:Stirling} we can find for each $i=1, \ldots, r$ constants $0<C^i_1<C_2^i$ (depending on $k_i$) such that 
\begin{align*}
C_1^i \binom{D_i(n)+k_i-1}{k_i}
= C_1^i \frac{\Gamma(D_i(n)+k_i)}{\Gamma(k_i+1)\Gamma(D_i(n))}
\leq  \frac{D_i^{k_i}(n)}{\Gamma(k_i+1)} \leq C_2^i \binom{D_i(n)+k_i-1}{k_i} 
\end{align*}
holds for all $n \geq (r-s) \vee 0$. 
Combining those constants leads to \eqref{eq:mart_power_form}. 
\end{proof}
Now to Proposition \ref{P:martingale_p_factorial}. The following proof adapts a martingale approach from \cite{mori05} (Theorem 2.1).
\begin{proof}[Proof of Proposition~\ref{P:martingale_p_factorial}]\(\)
\begin{enumerate}[label=\arabic*)]
\item
We call the total ball count at time \(m\) in the corresponding urn model
\[S_m:=z+m+\beta\cdot\Bigl\lfloor\frac{m}{l}\Bigr\rfloor\quad \text{with }z:=\sum_{i=1}^{N(0)}C_i(0)+j\cdot \beta.\]
Only one colour obtains a new ball at a time, so
\begin{align*}
&\E\Bigl(\prod_{i=1}^r\binom{C_i(m+1)+k_i-1}{k_i}\Big|\Fscr_m\Bigr)\\
&=\sum_{j=1}^r\E\Bigl(\underset{i\neq j}{\prod_{i=1}^r}\binom{C_i(m)+k_i-1}{k_i}\binom{C_j(m)+k_j}{k_j}\cdot\1_{\{C_j(m+1)=C_j(m)+1\}}\Big|\Fscr_m\Bigr)\\
&\;\;\;+\E\Bigl(\prod_{i=1}^r\binom{C_i(m)+k_i-1}{k_i}\cdot\1_{\bigcap_{j=1}^r\{C_j(m+1)=C_j(m)\}}\Big|\Fscr_m\Bigr)\\
\intertext{and using the transition probabilities of the urn model this equals}
& \;\;\; \sum_{j=1}^r\underset{i\neq j}{\prod_{i=1}^r}\binom{C_i(m)+k_i-1}{k_i}\cdot \binom{C_j(m)+k_j-1}{k_j}\cdot\frac{C_j(m)+k_j}{C_j(m)}\cdot\frac{C_j(m)}{S_m}\\
&\;\;\; +\prod_{i=1}^r\binom{C_i(m)+k_i-1}{k_i}\cdot \Bigl(1-\sum_{j=1}^r\frac{C_j(m)}{S_m}\Bigr)\\
&=\prod_{i=1}^r\binom{C_i(m)+k_i-1}{k_i}\Bigl(\sum_{j=1}^r\frac{C_j(m)+k_j}{S_m}+1-\sum_{j=1}^r\frac{ C_j(m)}{S_m}\Bigr)\\
&=\prod_{i=1}^r\binom{C_i(m)+k_i-1}{k_i}\frac{S_m+k}{S_m}.
\end{align*}
Now consider times \(m=n\cdot l\). These are the critical times where a new colour is migrated into the urn (a new vertex is added to the graph). In between nothing of interest happens and we can iterate the above calculation \(l\) times to obtain:
\[\E\Bigl(\prod_{i=1}^r\binom{C_i((n+1)\cdot l)+k_i-1}{k_i}\Big|\Fscr_{n\cdot l}\Bigr)=\prod_{i=1}^r\binom{C(n\cdot l)+k_i-1}{k_i}\prod_{i=0}^{l-1}\frac{S_{n\cdot l+i}+k}{S_{n\cdot l+i}}.\]
In the next step we see that
\begin{align*}
&\prod_{i=0}^{l-1}\frac{S_{n\cdot l+i}+k}{S_{n\cdot l+i}}=\prod_{i=0}^{l-1}\frac{z+n\cdot l+i+\beta\cdot n+k}{z+n\cdot l+i+\beta\cdot n}=\prod_{i=0}^{l-1}\frac{\frac{z+k+i}{l+\beta}+n}{\frac{z+i}{l+\beta}+n}\\
&=\biggl(\prod_{i=0}^{l-1}\frac{\Gamma(\frac{z+k+i}{l+\beta}+n+1)}{\Gamma(\frac{z+i}{l+\beta}+n+1)}\biggr)\Big/\biggl(\prod_{i=0}^{l-1}\frac{\Gamma(\frac{z+k+i}{l+\beta}+n)}{\Gamma(\frac{z+i}{l+\beta}+n)}\biggr)=:\frac{c(n+1,k)}{c(n,k)}
\end{align*}
and the martingale property follows. Lastly we apply Lemma \ref{L:Stirling} to each factor of \(c(n,k)\) to get
\begin{equation}\label{Eq:c(n,k):equiv}c(n,k)=\prod_{i=0}^{l-1}\frac{\Gamma(\frac{z+k+i}{l+\beta}+n)}{\Gamma(\frac{z+i}{l+\beta}+n)}\sim \prod_{i=0}^{l-1}n^{k/(l+\beta)}=n^{k\cdot l/(l+\beta)}.
\end{equation}
\item Doob's martingale convergence theorem (see Chapter XI.14 in \cite{doob94}) guarantees the existence of an almost sure limit of the martingale given in \eqref{term:process_p_factorial} and that this limit is in $L^1(\Pp)$. Lemma~\ref{L:ineq_mart_power_form} in combination with \eqref{Eq:c(n,k):equiv} then ensures the almost sure convergence of the sequence $n^{-k_i l/(l+\beta)}D_i(n)^{k_i}$ to a limit $\zeta_i^{k_i} \in L^1(\Pp)$ and a further application of Lemma~\ref{L:ineq_mart_power_form} implies that the expression in \eqref{term:limit_p_factorial} is indeed the limit of the process \eqref{term:process_p_factorial}.
\item
Doob's martingale convergence theorem (Chapter XI.14 in \cite{doob94}) yields that (\ref{term:limit_p_factorial}) is the right closure to the martingale (\ref{term:process_p_factorial}) if the latter is uniformly integrable. To show this uniform integrability we show that the process is $L^{1+\epsilon}$-bounded for any $\epsilon>0$. To this end, set \(p_i:=k_i(1+\epsilon)\) and \(p:=k(1+\epsilon)=p_1+\dots+p_r\). By Lemma \ref{L:ineq_mart_power_form} there exist constants \(C_1,C_2>0\) independent of \(n\) such that
\begin{align*}
\Bigl[\frac{1}{c(n,k)}\prod_{i=1}^r\binom{D_i(n)+k_i-1}{k_i}\Bigr]^{1+\epsilon}&\leq C_1\cdot\frac{\prod_{i=1}^rD^{p_i}_i(n)}{n^{p\cdot l/(l+\beta)}}\\
&\leq C_2\cdot \frac{1}{c(n,p)}\prod_{i=1}^r\binom{D_i(n)+p_i-1}{p_i}.
\end{align*}
Since the right-hand side is a multiple of a martingale, it has bounded expectation and the uniform integrability of (\ref{term:process_p_factorial}), and thus the first part of 3), follows.

This now allows us to iteratively trace back the expectation of \(\prod_{i=1}^r\zeta_i^{k_i}\) to the times at which colour $i, 1 \leq i \leq r,$ was first introduced to the urn, that is $n_0^i:=l(i-s) \vee 0$, and the number of balls was still deterministic, namely either \(\beta\) if \(i>N(0)\) or \(D_i(0)\) if \(i\leq N(0)\):
\begin{align*}
&\E\Bigl(\prod_{i=1}^r\frac{\zeta_i^{k_i}}{\Gamma(k_i+1)}\Bigr)=\E\Bigl(\E\Bigl(\prod_{i=1}^r\frac{\zeta_i^{k_i}}{\Gamma(k_i+1)}\Big|\Fscr_{n_0^r}\Bigr)\Bigr)\\
&=\E\Bigl(\frac{1}{c(n_0^r/l,k_1+\dots+k_r)}\prod_{i=1}^r\binom{D_i(n_0^r/l)+k_i-1}{k_i}\Bigr)\\
&=\frac{c(n_0^r/l,k_1+\dots+k_{r-1})}{c(n_0^r/l,k_1+\dots+k_r)}\cdot\begin{cases}\binom{D_r(0)+k_r-1}{k_r}\quad &\text{if }r\leq N(0)\\\binom{\beta+k_r-1}{k_r}\quad &\text{if }r> N(0)\end{cases}\\
&\phantom{=}\cdot\E\Bigl(\E\Bigl(\frac{1}{c(n_0^r/l,k_1+\dots+k_{r-1})}\prod_{i=1}^{r-1}\binom{D_i(n_0^r/l)+k_i-1}{k_i}\Big|\Fscr_{n_0^{r-1}}\Bigr)\Bigr)\\
&=\prod_{i=2}^{r}\frac{c(n_0^i/l,k_1+\dots+k_{i-1})}{c(n_0^i/l,k_1+\dots+k_i)}\cdot\frac{1}{c(n_0^1/l,k_1)}\prod_{i=1}^r\begin{cases}\binom{D_i(0)+k_i-1}{k_i}\quad &\text{if }i\leq N(0)\\\binom{\beta+k_i-1}{k_i}\quad &\text{if }i> N(0).\end{cases}
\end{align*}
\end{enumerate}
\end{proof}

We end this section with a uniform integrability property that will turn out to be useful on multiple occasions. 
\begin{cor}\label{L:uniformInt}
Under the assumptions of and with the notation from Proposition~\ref{P:martingale_p_factorial} the process \(\bigl(n^{-\frac{k\cdot l}{l+\beta}}\prod_{i=1}^rD_i^{k_i}(n)\bigr)_n\) is uniformly integrable.
\end{cor}
\begin{proof}
This follows from the uniform integrability of the martingale \eqref{term:process_p_factorial} shown in Proposition~\ref{P:martingale_p_factorial} \textit{1)} and \textit{3)} together with Lemma~\ref{L:ineq_mart_power_form}. 
\end{proof} 

\subsection{Proof of Theorem~\ref{Breiman_generalised}}\label{Sec:BreimanProof}
We adapt the proof from Theorem 3 in \cite{wang21} to the case of a general separable Banach space. By our additional assumption about the continuity of $(X_t)$ we are able to straighten a previously not properly addressed subtlety in the proof about the application of Egorov's theorem. 

\begin{proof}[Proof of Theorem~\ref{Breiman_generalised}] By a Portmanteau theorem for $\mathbb{M}$-convergence, see Theorem 2.1 in \cite{lindskog14}, it is sufficient to show 
\begin{equation}\label{Eq:ToShowmain}\lim_{t\to\infty}t\E f\Bigl(\frac{X_T T}{b(t)}\Bigr)=\int_{0}^\infty\E f(X_\infty y)\,\nu_\alpha(\dint y)<\infty,\end{equation}
for all non-negative, bounded and uniformly continuous functions $f$ on \(B\) whose support is bounded away from $0$, i.e. there exists an \(\epsilon_0>0\) such that \(f(x)=0\) for all $x \in B_{\epsilon_0}(0)$. Without loss of generality we will assume that $f$ is bounded by 1.
In order to show \eqref{Eq:ToShowmain} we write, for $\eta, M >0$,
\begin{align*}
&\Bigl|t\E f\Bigl(\frac{X_TT}{b(t)}\Bigr)-\int_0^\infty\E f(X_\infty y)\nu_\alpha(\dint y)\Bigr|\\
&\leq\Bigl|t\E\Bigl(f\Bigl(\frac{X_TT}{b(t)}\Bigr)\1_{ \{M \geq T/b(t)\geq\eta\}}\Bigr)-\int_\eta^M\E f(X_\infty y)\nu_\alpha(\dint y)\Bigr| \\
&+t\E\Bigl(f\Bigl(\frac{X_TT}{b(t)}\Bigr)\1_{\{T/b(t)<\eta\}}\Bigr) +\int_0^\eta\E f(X_\infty y)\nu_\alpha(\dint y) \\
&+t\E\Bigl(f\Bigl(\frac{X_TT}{b(t)}\Bigr)\1_{\{T/b(t)>M\}}\Bigr) +\int_M^\infty\E f(X_\infty y)\nu_\alpha(\dint y) \\
&=: I^1(t,\eta,M)+I^2(t,\eta)+I^3(t,M).
\end{align*}
\begin{enumerate}
\item We start by showing $I^1(t,\eta,M) \to 0, t \to \infty.$ To this end, first fix some \(\epsilon>0\). By Egorov's theorem there exists a measurable set \(A_\epsilon\) with \(\Pp(A_\epsilon^c)<\epsilon\) such that \(X_t\to X_\infty\) for \(t\to\infty\) uniformly on \(A_\epsilon\). At this point we use the one-sided continuity of \((X_t)\) as a sufficient condition to ensure the measurability of \(A_\epsilon\). This allows us to bound \(I^1(t,\eta,M)\) separately on \(A_\epsilon\) and \(A_\epsilon^c\):
\begin{align*}
I^1(t,\eta,M)&\leq t\E\Bigl(f\Bigl(\frac{X_TT}{b(t)}\Bigr)\1_{A_\epsilon^c \cap \{M \geq T/b(t)\geq\eta\}}\Bigr)+\int_\eta^M\E \Bigl(f(X_\infty y)\1_{A_\epsilon^c}\Bigr) \nu_\alpha(\dint y) \\
& + t\E\Bigl(\Bigl|f\Bigl(\frac{X_T T}{b(t)}\Bigr)-f\Bigl(\frac{X_\infty T}{b(t)}\Bigr)\Bigr|\1_{A_\epsilon\cap\{M \geq T/b(t)\geq\eta\}}\Bigr)\\
& + \Bigl|t\E\Bigl(f\Bigl(\frac{X_\infty T}{b(t)}\Bigr)\1_{A_\epsilon\cap\{M \geq T/b(t)\geq\eta\}}\Bigr)-\int_\eta ^M \E\Bigl( f(X_\infty y)\1_{A_\epsilon}\Bigr) \nu_\alpha(\dint y)\Bigr|\\
&=: I_a^1(t,\eta,M)+I_b^1(t,\eta,M)+I_c^1(t,\eta,M)+I_d^1(t,\eta,M).
\end{align*}
Starting with $I_a^1(t,\eta,M)$ we use that \(A_\epsilon^c \in\sigma(X_t,t\geq 0)\) which therefore is independent of \(T\):
\begin{align*}
\limsup_{t \to \infty} I_a^1(t,\eta,M) \leq \limsup_{t \to \infty}  t\Pp\Bigl(A_\epsilon^c \cap\Bigl\{\frac{T}{b(t)}\geq \eta \Bigr\}\Bigr)\leq \epsilon\cdot\eta^{-\alpha}.
\end{align*}
Second,
\begin{align*}
I_b^1(t,\eta,M) \leq \Pp (A_\epsilon^c)\int_\eta^M \nu_\alpha(\dint y) \leq \epsilon \cdot (\eta^{-\alpha}-M^{-\alpha}).
\end{align*}
For $I_c^1(t,\eta,M)$, we note that \(b(t)\to\infty\) for \(t\to\infty\), yielding an arbitrarily large lower bound for \(T\) on the set \(\{T/b(t)\geq\eta\}\). Then, intersecting with \(A_\epsilon\) we obtain that \(b(t)^{-1} T \norm{X_T - X_\infty }\cdot\1_{A_\epsilon\cap\{M \geq T/b(t)\geq\eta\}}\) uniformly tends to \(0\) and combining this with the uniform continuity of \(f\) we find a \(c(t)\) with \(c(t)\to 0\) for \(t\to\infty\) such that, for $t \to \infty$,
\begin{align*}
I_c^1(t,\eta,M)&\leq c(t)\cdot t\Pp\Bigl(A_\epsilon\cap\Bigl\{M \geq \frac{T}{b(t)}\geq \eta \Bigr\}\Bigr)\leq c(t)\cdot t\Pp\Bigl(\frac{T}{b(t)}\geq\eta \Bigr)\to 0.
\end{align*}
Finally, for $I_d^1(t,\eta,M)$ we observe that 
$$ y \mapsto E(f(X_\infty y)\mathds{1}_{A_\epsilon}) \mathds{1}_{[\eta,M]}(y)$$ 
is bounded and has support bounded away from 0. Even though it is not continuous in $y$, a standard approximation argument by continuous functions combined with dominated convergence gives
\begin{eqnarray*}
  &&  t\E\Bigl(f\Bigl(\frac{X_\infty T}{b(t)}\Bigr)\1_{A_\epsilon\cap\{M \geq T/b(t)\geq\eta\}}\Bigr) \\
   & = &  \int_\eta^M \E(f(X_\infty y )\1_{A_\epsilon})t\Pp^{T/b(t)}(\dint y)  \to \int_\eta ^M \E (f(X_\infty y)\1_{A_\epsilon}) \nu_\alpha(\dint y)
\end{eqnarray*}
for $t \to \infty$, where we used regular variation of $T$ in combination with Fubini's theorem as $T$ is independent of $X_\infty$. Thus, $I_d(t,\eta,M) \to 0$ for  $t \to \infty$.
From the above and since we can choose $\epsilon>0$ arbitrarily small, we see that $I^1(t,\eta,M) \to 0, t \to \infty$.
\item We split up $I^2(t,\eta)$ into 
$$ t\E\Bigl(f\Bigl(\frac{X_TT}{b(t)}\Bigr)\1_{\{T/b(t)<\eta\}}\Bigr) +\int_0^\eta\E f(X_\infty y)\nu_\alpha(\dint y)=:I^2_a(t,\eta)+I^2_b(\eta).$$
Then, using Markov's inequality as well as the independence of \((X_t)_{t\geq 0}\) and \(T\), we arrive at the upper bound
\begin{align*}
&I^2_a(t,\eta) \leq t \Pp\Bigl(\Big\|\frac{X_TT}{b(t)}\Bigr\|>\epsilon_0,\frac{T}{b(t)}<\eta\Bigr)\\
&=t\Pp\Bigl(\Big\|\frac{X_TT}{b(t)}\1_{\{T/b(t)<\eta\}}\Bigr\|>\epsilon_0 \Bigr)\leq  \epsilon_0^{-\alpha'}t\E\Bigl(\Big\|\frac{X_TT}{b(t)}\1_{\{T/b(t)<\eta\}}\Bigr\|^{\alpha'}\Bigr)\\
&= \epsilon_0^{-\alpha'}\int_{(0,\eta)}\E\norm{X_{b(t)y}}^{\alpha'} y^{\alpha'}t\Pp(T/b(t)\in\dint y)\\
&\leq \epsilon_0^{-\alpha'}  \sup_{t\geq 0}\E\norm{X_t}^{\alpha'} t\E\Bigl(\Bigl(\frac{T}{b(t)}\Bigr)^{\alpha'}\1_{\{T/b(t)<\eta\}}\Bigr)
\end{align*}
By Karamata's Theorem applied to truncated moments of regularly varying random variables, see Proposition 1.4.6 in \cite{kulik20}, we get
\begin{align*}
 & t\E\Bigl(\Bigl(\frac{T}{b(t)}\Bigr)^{\alpha'}\1_{\{T/b(t)<\eta\}}\Bigr) \\
 &= t\Pp(T>b(t))\frac{\E\Bigl(T^{\alpha'}\mathds{1}_{\{T<\eta b(t)\}}\Bigr)}{(b(t))^{\alpha'}\Pp(T>b(t))} \\
 &\to \frac{\alpha}{\alpha'-\alpha}\eta^{\alpha'-\alpha}.
\end{align*}
Thus, for $t \to \infty, \eta \to 0$, $I^2_a(t,\eta) \to 0$. 

In order to bound $I^2_b(\eta)$ we first note that \(\E\norm{X_\infty}^{\beta}\) is finite for all $\alpha<\beta<\alpha'$ because \((\norm{X_t}^\beta)_t\) is uniformly integrable (as \(\sup_t\E\bigl[(\norm{X_t}^\beta)^{\frac{\alpha'}{\beta}}\bigr]<\infty\) and \(\frac{\alpha'}{\beta}>1\)) and the Vitali convergence theorem implies that \(\norm{X_t}^\beta\) converges in \(L^1\) to \(\norm{X_\infty}^\beta\) and so do the expected values. Then, again by Markov's inequality we get
\begin{align*}
& \int_0^\eta\E f(X_\infty y)\nu_\alpha(\dint y) \leq \int_0^\eta \Pp(\|X_\infty\|> \epsilon_0 /y)\nu_\alpha(\dint y) \\
&\leq \int_0^\eta \E\bigl(\|X_\infty\|^\beta\bigr) \epsilon_0^{-\beta} y^\beta \nu_\alpha(\dint y) \\
&\leq \E\bigl(\|X_\infty\|^\beta\bigr) \epsilon_0^{-\beta} \frac{\alpha}{\beta-\alpha}\eta^{\beta-\alpha}.
\end{align*}
Again, with $\eta \to 0$ we have shown that $I^2_b(\eta) \to 0$ and therefore $I^2(t,\eta) \to 0$ and also that the right hand side in \eqref{Eq:ToShowmain} is finite, since $\int_\eta^\infty \E f(X_\infty y)\nu_\alpha(\dint y) \leq \eta^{-\alpha}$.
\item Finally, regarding $I^3(t,M)$ we note that 
\begin{align*}
 & \limsup_{t \to \infty} I^3(t,M) \leq \limsup_{t \to \infty} t\Pp(T>M b(t)) + \int_M^\infty \nu_\alpha (dy) \leq 2 M^{-\alpha},
\end{align*}
and so $\lim_{M \to \infty} \limsup_{t \to \infty} I^3(t,M)=0$. 
This finishes the proof of \eqref{Eq:ToShowmain}.
\end{enumerate}
\end{proof}
\begin{remark}
One could do without the continuity assumption of \((X_t)\) provided the measurability of all sets in the proof is ensured. In particular, this refers to the \(A_\epsilon\) from Egorov's theorem.
\end{remark}

\subsection{Moment conditions}
The following statement extends Lemma~4.1 in \cite{pekoz16} to arbitrary values of $\beta \geq 0$.

\begin{lemma}\label{L:MomentCond}
Assume the urn model of Section~\ref{sec:InCoUrn}, fix $k \in \mathbb{N}$ and let
\[U_k(n):=C_1(n)+\dots+C_k(n)\]
denote the total number of balls of colours 1 to k at time \(n\). Then for every \(q\in\N\) there exist \(c,C>0\) such that for all \(n\in\N_0\)
\begin{equation} \label{eq:MomentCond}
cn^{ql/(l+\beta)}<\E(U_k(n)^q)<Cn^{ql/(l+\beta)}.
\end{equation}
\end{lemma}
\begin{proof}
Let $k \in \mathbb{N}$ be fixed throughout this proof. Note first that $U_k(n), n \in \mathbb{N}_0,$ is deterministic until there are at least $k+1$ different colours in the urn, i.e.\ for $n \leq n_0$ with
\[n_0:=l(k+1-s) \vee 0\] 
and set
\[n_p:=n_0+p,\,p\in\N.\]
It is then sufficient to show \eqref{eq:MomentCond} for $n=n_p, p \in \mathbb{N},$ as for those finitely many \(n \leq n_0\) there exist trivial bounds.
Let \(h_p=U_{(k+1)\vee s}(n_0)+p+\beta\floor{\frac{p}{l}}\) be the (deterministic) total number of balls at time \(n_p\) and define \begin{equation} M_{p,q}:=\prod_{i=0}^{q-1}(i+U_k(n_p)), p \in \mathbb{N}_0, q \in \mathbb{N}.\end{equation} We start the proof by showing that
\begin{equation}\label{eq:RepDp}
\E M_{p,q}=M_{0,q}\cdot\prod_{i=0}^{p-1}\Bigl(1+\frac{q}{h_i}\Bigr).
\end{equation}
Given the value \(U_k(n_{p-1})\), there are only two possible values for \(U_k(n_p)\), which are \(U_k(n_{p-1})+1\) and \(U_k(n_{p-1})\): Either we draw a ball of colours 1 to $k$ or we do not. Using this, one gets:
\begin{align*}
\E(M_{p,q}|U_k(n_{p-1}))&=\frac{U_k(n_{p-1})}{h_{p-1}}\prod_{i=0}^{q-1}(i+U_k(n_{p-1})+1)+\frac{h_{p-1}-U_k(n_{p-1})}{h_{p-1}}M_{p-1,q}\\
&= \frac{U_k(n_{p-1})}{h_{p-1}}\prod_{i=1}^{q}(i+U_k(n_{p-1}))+\frac{h_{p-1}-U_k(n_{p-1})}{h_{p-1}}M_{p-1,q}\\
&=\frac{U_k(n_{p-1})}{h_{p-1}}\frac{M_{p-1,q}(U_k(n_{p-1})+q)}{U_k(n_{p-1})}+\frac{h_{p-1}-U_k(n_{p-1})}{h_{p-1}}M_{p-1,q}\\
&=M_{p-1,q}\Bigl(1+\frac{q}{h_{p-1}}\Bigr).
\end{align*}
By iterating this calculation \(p\) times, one obtains (\ref{eq:RepDp}).\\
Now, by definition of $h_p, p \in \mathbb{N}$, we have
\[U_{(k+1)\vee s}(n_0)+p+\beta\Bigl(\frac{p}{l}-1\Bigr)\leq h_p\leq U_{(k+1)\vee s}(n_0)+p+\beta\frac{p}{l}.\]
Setting \(x:=\frac{l}{l+\beta}\) and \(y:=(U_{(k+1)\vee s}(n_0)-\beta)\frac{l}{l+\beta}\) (which is non-negative as for each colour there are at least \(\beta\) balls) for abbreviation this is equivalent to
\begin{align}
\frac{y+p}{x} &\leq h
_p\leq \frac{y+p+\beta x}{x}.\label{eq:nyx}
\end{align}
For the upper bound in \eqref{eq:MomentCond} use (\ref{eq:RepDp}) and (\ref{eq:nyx}) to get
\begin{align*}
\E M_{p,q}\leq M_{0,q} \cdot\prod_{i=0}^{p-1}\Bigl(1+\frac{qx}{y+i}\Bigr) 
&= M_{0,q}  \cdot \frac{\Gamma(y)}{\Gamma(qx+y)} \cdot\frac{\Gamma(qx+y+p)}{\Gamma(y+p)}.
\end{align*}
Thus, using Lemma~\ref{L:Stirling}, there exist $\tilde{C}, C>0$ such that
\begin{align*}
\E U_k(n_p)^q&\leq \E M_{p,q} \leq M_{0,q}  \cdot \frac{\Gamma(y)}{\Gamma(qx+y)} \cdot\frac{\Gamma(qx+y+p)}{\Gamma(y+p)} \\
&\leq \tilde{C} (y+p)^{qx} < C n_p^{qx}= C n_p^{ql/(l+\beta)}
\end{align*}
for all $p \in \mathbb{N}$, proving the upper bound. 

For the lower bound, use Jensen's inequality and \eqref{eq:RepDp}, \eqref{eq:nyx}, Lemma~\ref{L:Stirling} and similar reasoning as above to see that there exist $\tilde{c}, c>0$ such that
\begin{align*}
    \E U_k(n_p)^q&=\E M_{p,1}^q\geq(\E M_{p,1})^q \\
    &\geq \left(M_{0,1}\cdot\prod_{i=0}^{p-1}\Bigl(1+\frac{x}{y+i+\beta x}\Bigr) \right)^q \\
    &= M_{0,1}^q \left(\frac{\Gamma(y+\beta x)}{\Gamma(x+y+\beta x)}\right)^q \left(\frac{\Gamma(x+y+\beta x+p)}{\Gamma(y+\beta x+p)}\right)^q \\
    &\geq \tilde{c} (y+\beta x +p)^{qx} > c n_p^{qx}=c n_p^{ql/(l+\beta)},
\end{align*}
for all $p \in \mathbb{N}$, which finishes the proof.

\end{proof}

\subsection{Proof of Lemma \ref{L:Hilfs_Beta}}\label{Sec:BetaIndProof}
\begin{proof}[Proof of Lemma \ref{L:Hilfs_Beta}]
\begin{enumerate}[label=\arabic*)]
\item We adapt the proof of Lemma~3.3 from \cite{mori05}. By definition $B_k^\downarrow=\zeta_{k}/(\zeta_1+\dots+\zeta_{k})$ and so the statement follows immediately for $k=1$. Thus, keep $k>1$ fixed in the following.
We disregard all colours in the urn model but \(1\) through \(k\) and further simplify the model: To find the distribution of $B_k^\downarrow$ it is sufficient to view all balls of colours $1$ to $k-1$ as being \enquote{black} and those of colour \(k\) as being \enquote{white}. At the time \(n_0^k\), when the last relevant colour is added to the urn, the number $w_k$ of white and $b_k$ of black balls is deterministic and given by
$$ (w_k,b_k)=\begin{cases}(C_k(0),\sum_{i=1}^{k-1}C_i(0)) \quad&\text{if } 1<k\leq s\\ (\beta,\sum_{i=1}^{s}C_i(0)+(l+\beta)(k-1-s)+l) \quad &\text{if } k>s\end{cases}. $$
Now at each time either the number of black and white balls does not change or a new ball is added to one of them. Conditionally on the latter, the transition probabilities are the same as in a traditional P{\'o}lya urn and therefore Lemma \ref{L:Polya_Dir} yields that the vector \((B_k^\downarrow,1-B_k^\downarrow)\) has a Dirichlet distribution, whose marginals are known to be beta distributions with the given parameters.

\item We provide an alternative construction for $C^r(n)$ to help us derive the desired independence properties:\\
We start with the usual setup of an urn with \(s\) different colours and starting amounts \(C_0(0),\dots,C_s(0)\). From there on, instead of picking a ball directly, in each time step we flip a sequence of independent Bernoulli coins, determined below, until we observe heads for the first time. Then we add a new ball of one colour, based on which flip resulted in heads, and proceed to the next time step. The individual coins may have different success probabilities, i.e. probabilities of showing heads. More precisely, for the ball to be selected at time \(n\), let \(m\) be the currently largest numbered colour in the urn. If \(m>r\) the first flip in a time step always represents the colours \(>r\). If that coin flip shows heads we add a new ball to one of those (which one specifically is irrelevant to us). If it shows tails, however, we start traversing the remaining \(r\) colours in backwards order. This means we start by flipping for colour \(r\), then either add a new ball to it if the flip showed heads or otherwise continue to \(r-1\) and so on. If \(m\leq r\) we skip the flips for colours \(m+1,m+2,\dots,\)``\(>r\)''. To also incorporate the migration of new colours in this model, after every \(l\)th time step, we add additional \(\beta\) balls to colour \(m+1\) or if \(m\geq r\) to the ``\(>r\)''-category. Next, for this process to have the same distribution as $(C^r(n))_{n \in \mathbb{N}_0}$, we need to find suitable success probabilities \(p_i(n)\) when flipping for colour \(i\) in time step \(n\mapsto n+1\). To this end, we write 
$$ U_k(n):=\sum_{i=1}^k C_i(n), \;\;\; k \geq 1,$$
for abbreviation and set
\begin{align*}
\mbox{if }i\leq r:&&p_i(n)&=\frac{C_i(n)}{C_1(n)+\dots+C_i(n)}=1-\frac{U_{i-1}(n)}{U_i(n)}, \\
\mbox{ and } && p_{>r}(n)&=\frac{C_{r+1}(n)+\dots+C_m(n)}{C_1(n)+\dots+C_m(n)}=1-\frac{U_{r}(n)}{U_m(n)}.
\end{align*}
To verify that this yields exactly the same transition probabilities as in the urn model, write (with \(\Fscr_n=\sigma(C_i(t):\,t\leq n,\,i,t\in\N)\)):\\
for colours \(i\leq r\):
\begin{align*}
&\Pp(C_i(n+1)-C_i(n)=1|\Fscr_{n})=\frac{C_i(n)}{U_m(n)}=\frac{U_{r}(n)}{U_m(n)} \frac{U_{r-1}(n)}{U_r(n)}\dots \frac{U_i(n)}{U_{i+1}(n)}\frac{C_i(n)}{U_i(n)}\\
&\qquad=(1-p_{>r}(n))\cdot (1-p_r(n))\cdot\hdots\cdot (1-p_{i+1}(n))\cdot p_{i}(n)
\end{align*}
and for colour \(``>r"\):
$$ \Pp([U_{m}(n+1)-U_r(n+1)]-[U_m(n)-U_r(n)]=1|\Fscr_{n})=\frac{U_m(n)-U_r(n)}{U_m(n)}=p_{>r}(n).  $$
So we conclude that the two processes indeed have the same transition probabilities and proceed to show the independence of the $B_i^\downarrow $'s.

For that let \((Y_j^i)_{j=1}^\infty\) be the random variables which contain the outcome of the \(j\)th flip for color \(i\), \(i=1,\dots,r\). Define \(T_j^i\) to be the step in the above urn model when coin \(Y_j^i\) is flipped, noting that \(T_j^i,\,j>0\) is random but a.s. finite by Lemma \ref{L:InftyDegree}. The success probability \(p_i(n)\) at time $n=T_j^i$ is dependent on prior flips and given by 
$$ \frac{C_i(n)}{C_1(n)+\dots+C_i(n)}=\frac{C_i(T^i_1)+\sum_{k=1}^{j-1}Y_k^i}{U_i(T^i_1)+j-1} $$
for colours $i \leq r$, as $U_i(T^i_1)$ balls of colours 1 to \(i\) and \(C_i(T^i_1)\) balls of colour \(i\) are present when we first flip the coin for this colour and at the $j$-th flip further $j-1$ balls of colours 1 to \(i\) have already been added to the urn, with $\sum_{k=1}^{j-1}Y_k^i$ of them to colour $i$. Thereby, the $p_i(n)$ at time $n=T_j^i$ are completely determined by deterministic starting values $U_i(T^i_1),C_i(T^i_1)$ and the previous flips for colour \(i\) and the sequences \((Y_j^1)_j, \ldots, (Y_j^r)_j\) are jointly independent.\\
With this, the independence of the \(B_i^\downarrow\)'s follows since for $j \to \infty$,
\begin{align*}
\label{Eq:IndBis}\frac{C_i(T^i_1)+\sum_{k=1}^{j-1}Y_k^i}{U_i(T^i_1)+j-1}=\frac{C_i(T_j^i)}{C_1(T_j^i)+\dots+C_i(T_j^i)} \to B_i^\downarrow \qquad\quad\text{ for }i=2,\dots,r.
\end{align*}
\item 
In addition to the sequences introduced in 2), let \((Y_j^{>r})_{j=1}^\infty\) be the random variables which contain the outcome of the \(j\)th flip for color \(``>r"\) and let $T_j^{>r}, j>0$ be the step in the urn model when coin $Y_j^{>r}$ is flipped. The first flip $T_1^{>r}$ takes place at time $n_0^{r+1}$ with $U_{(r+1) \vee s}(n_0^{r+1})$ total balls in the urn, and $U_{(r+1) \vee s}(n_0^{r+1})-U_{r}(n_0^{r+1})$ of them of colour \(``>r"\), both numbers being deterministic. After $n_0^{r+1}$, the coin for colour \(``>r"\) is flipped in every consecutive step of the urn model. Thus, at the $j$-th flip for this colour, the total number of balls has grown to $U_{(r+1) \vee s}(n_0^{r+1})+j-1+\beta\floor{\frac{j-1}{l}}$ (where the last summand is due to the immigration), and the number of colour \(``>r"\) has grown to $U_{(r+1) \vee s}(n_0^{r+1})-U_{r}(n_0^{r+1})+\sum_{k=1}^{j-1}Y_k^{>r}+\beta\floor{\frac{j-1}{l}}$. Therefore, the success probability $p_{>r}(n)$ at time $n=T_j^{>r}$ is given by
$$ \frac{U_{(r+1) \vee s}(n_0^{r+1})-U_{r}(n_0^{r+1})+\sum_{k=1}^{j-1}Y_k^{>r}+\beta\floor{\frac{j-1}{l}}}{U_{(r+1) \vee s}(n_0^{r+1})+j-1+\beta\floor{\frac{j-1}{l}}}, $$
and the sequence \((Y_j^{r>})_{j=1}^\infty\) is independent of the sequences introduced in 2). 

Observe that
\[\frac{1}{j^{l/(l+\beta)}}U_r(n_0^{r+1}+j)=\frac{1}{j^{l/(l+\beta)}}(U_r(n_0^{r+1})+\sum_{k=1}^j(1-Y_k^{>r}))\overset{j\to\infty}{\to}\zeta_1+\dots+\zeta_r,\]
and by independence of the coin flip sequences the limit $\zeta_1+\dots+\zeta_r$ is thus independent of \((B_1^\downarrow,\dots,B_r^\downarrow)\) and thus from \((B_1,\dots,B_r)\) as well.
\end{enumerate}
\end{proof}

\subsection{Proof of Proposition~\ref{P:BreiCond}}\label{Sec:BreiCondProof}
\begin{proof}[Proof of Proposition~\ref{P:BreiCond}]
\phantom{}
We start with showing 2). 
\begin{enumerate}
\item[2)] Case \(\alpha \leq p \): Let \(\alpha'>p\geq \alpha\), then by H\"older's inequality
\[\sup_n\E\Bigl(\Bigl\lVert\frac{D(n)}{n^{l/(l+\beta)}}\Bigr\rVert_p^\alpha\Bigr)\leq \sup_n\Bigl(\E\Bigl(\Bigl\lVert\frac{D(n)}{n^{l/(l+\beta)}}\Bigr\rVert_p^{\alpha'}\Bigr)\Bigr)^{\frac{\alpha}{\alpha'}}=\Bigl(\sup_n\E\Bigl(\Bigl\lVert\frac{D(n)}{n^{l/(l+\beta)}}\Bigr\rVert_p^{\alpha'}\Bigr)\Bigr)^{\frac{\alpha}{\alpha'}}\]
and this case follows whenever there is a bound for the next one.\\
Case \(\alpha > p\): Because of the inverse Minkowski's inequality (see, e.g., III 2.4 Theorem 9 in \cite{bullen03}), \(p\)-quasinorms with \(0<p<1\) are concave on \(\R_{\geq 0}^d\). We can thus apply Jensen's inequality to
\[\E\Bigl(\Bigl\lVert\frac{D(n)}{n^{l/(l+\beta)}}\Bigr\rVert_p^\alpha\Bigr)=\E\Bigl(\Bigl\lVert\Bigl(\frac{D(n)}{n^{l/(l+\beta)}}\Bigr)^\alpha\Bigr\rVert_{\frac{p}{\alpha}}\Bigr)\leq \Bigl(\sum_{i=1}^{N(n)}\Bigl(\E\frac{D_i^\alpha(n)}{n^{\alpha\cdot l/(l+\beta)}}\Bigr)^{\frac{p}{\alpha}}\Bigr)^{\frac{\alpha}{p}}.\]
An application of Proposition \ref{P:martingale_p_factorial}, 1) and Lemma \ref{L:ineq_mart_power_form} yields that there exists a $C>0$ such that the right hand side is bounded above by
\begin{align*}
C\cdot\Biggl(\sum_{i=1}^{N(n)}\Biggl(c(n,\alpha)^{-1}\E {D_i(n)+\alpha-1\choose \alpha}\Biggr)^{\frac{p}{\alpha}}\Biggr)^{\frac{\alpha}{p}}.
\end{align*}
Now, each base in the summands on the right is the expectation of a closed martingale and by Proposition \ref{P:martingale_p_factorial}, 3) 
\begin{align}\label{Eq:ZetaNormSum}
\Biggl(\sum_{i=1}^{N(n)}\Biggl(c(n,\alpha)^{-1}\E {D_i(n)+\alpha-1\choose \alpha}\Biggr)^{\frac{p}{\alpha}}\Biggr)^{\frac{\alpha}{p}}=\Bigl(\sum_{i=1}^{N(n)}(\E \zeta_i^\alpha)^{\frac{p}{\alpha}}\Bigr)^{\frac{\alpha}{p}}.
\end{align}
By \eqref{eq:expec_zetap} and Proposition \ref{P:martingale_p_factorial}, 1) we have, for any fixed $k \in R_{\geq 0}$ and all $i>N(0)$
\begin{align} \label{Eq:ZetaMomentsAsymp}
\E (\zeta_i^k) = C(k) c((i-s),k)^{-1} \sim C(k) (i-s)^{-k\cdot l/(l + \beta)},
\end{align}
where $C(k)$ denotes a constant that does not depend on $i$. Thus, the limit as $n \to \infty$ of the right hand side of \eqref{Eq:ZetaNormSum} is finite if and only if
\begin{align*} 
\frac{\alpha \cdot l}{l+\beta} \cdot \frac{p}{\alpha} > 1 \;\;\; \Leftrightarrow \;\;\;  p > \frac{l+ \beta}{l}, 
\end{align*}
which is guaranteed by our assumption. 
\item[1) (a)] Assume first that $\frac{l+\beta}{l}<p<\infty$. Use monotone convergence to get
\begin{align}\label{Eq:MonConvEZeta}
\E(\|\zeta\|_p^p)=\E\sum_{i=1}^\infty\zeta_i^p =\sum_{i=1}^\infty\E\zeta_i^p.
\end{align}
Again from \eqref{Eq:ZetaMomentsAsymp}, this sum is finite under our assumption that $p>(\l+\beta)/l$ and so $\|\zeta\|_p^p<\infty $ a.s. Furthermore, since $\| \zeta\|_\infty \leq \| \zeta\|_p$, the result also follows for $p=\infty$.
\item[1) (b)]
Next we prove the almost sure convergence in \(c_{\norm{\cdot}_p}\). We need to show
\[\Bigl\lVert\frac{D(n)}{n^{l/(l+\beta)}}-\zeta\Bigr\rVert_p^p=\sum_{i=1}^\infty\Bigl|\frac{D_i(n)}{n^{l/(l+\beta)}}-\zeta_i\Bigr|^p\to 0,\,n\to\infty\quad\text{a.s.}\]
By Scheff{\'e}'s theorem the above convergence follows from componentwise convergence, i.e.
 \begin{align}\label{Eq:Pointwiselp} \frac{D_i(n)}{n^{l/(l+\beta)}} \to \zeta_i, \, n \to \infty, \quad\text{a.s.} \end{align}
 for all $i \in \mathbb{N}$ in combination with convergence of the norms, which is
\begin{align}\label{Eq:Normconvlp}\sum_{i=1}^{N(n)}\frac{D_i(n)^p}{n^{pl/(l+\beta)}} \to\sum_{i=1}^\infty\zeta_i^p,\,n\to\infty. \end{align}
Proposition \ref{P:martingale_p_factorial}, 2) implies \eqref{Eq:Pointwiselp}, so we are left to show \eqref{Eq:Normconvlp}. 
To this end, introduce the random variables \(X_i(n):=c(n,p)^{-1}{D_i(n)+p-1\choose p}\), which form, for each $i \in \mathbb{N}$, a martingale, according to Proposition \ref{P:martingale_p_factorial}, 1). Their cumulative sums form submartingales, since
\begin{align*}
\E\Bigl(\sum_{i=1}^{N(n)}X_i(n)\Big|\Fscr_{(n-1)l}\Bigr)=\sum_{i=1}^{N(n)}\E(X_i(n)|\Fscr_{(n-1)l})\geq\sum_{i=1}^{\mathclap{N(n-1)}}X_i(n-1), \; i \in \mathbb{N}.
\end{align*}
By Lemma \ref{L:ineq_mart_power_form} there exists for each $k \in \mathbb{N}$ a \(C\) such that
\begin{align*}
\E\Bigl(\Bigl(\sum_{i=1}^{N(n)}X_i(n)\Bigr)^k\Bigr)\leq C\E\Bigl(\sum_{i=1}^{N(n)}\frac{D_i(n)^p}{n^{pl/(l+\beta)}}\Bigr)^k.
\end{align*}
By \ref{P:BreiCond:moments} with \(\alpha=pk\), we observe that the right hand side is bounded uniformly in \(n\). This means the submartingale is uniformly integrable and therefore convergent almost surely and in \(L_1\) (Chapter XI.14 in \cite{doob94}). To derive the limit we first note that by \ref{P:martingale_p_factorial}, 2)
\[\lim_{n \to \infty} \sum_{i=1}^{N(n)}X_i(n) - \sum_{i=1}^l \frac{\zeta_i^p}{\Gamma(p+1)}=\lim_{n \to \infty} \sum_{i=l+1}^{N(n)}X_i(n) \geq 0\]
for all $l \in \mathbb{N}$. Let $l \to \infty$ to conclude that
\[ \lim_{n \to \infty} \sum_{i=1}^{N(n)}X_i(n) - \sum_{i=1}^\infty \frac{\zeta_i^p}{\Gamma(p+1)} \geq 0.\]
This implies that 
\begin{align*}
&E\Bigl( \Bigl| \lim_{n \to \infty} \sum_{i=1}^{N(n)}X_i(n) - \sum_{i=1}^\infty \frac{\zeta_i^p}{\Gamma(p+1)} \Bigr| \Bigr) \\
&= E\Bigl( \lim_{n \to \infty} \sum_{i=1}^{N(n)}X_i(n) - \sum_{i=1}^\infty \frac{\zeta_i^p}{\Gamma(p+1)} \Bigr)  \\
&= \lim_{n \to \infty} \sum_{i=1}^{N(n)}E(X_i(n))- \sum_{i=1}^\infty E\Bigl(\frac{\zeta_i^p}{\Gamma(p+1)} \Bigr) = 0,
\end{align*}
where we used the $L_1$-convergence of our submartingale together with \eqref{Eq:MonConvEZeta} in the penultimate step and \ref{P:martingale_p_factorial}, 2) in the last step. Thus, 
\begin{align} \lim_{n \to \infty} \sum_{i=1}^{N(n)}c(n,p)^{-1}{D_i(n)+p-1\choose p} =  \sum_{i=1}^\infty \frac{\zeta_i^p}{\Gamma(p+1)} \;\;  \mbox{a.s.} \end{align}
Returning to the original process, we can use Proposition \ref{P:martingale_p_factorial}, 1) and Lemma \ref{L:ineq_mart_power_form} to find a $C>0$ such that 
\begin{align*} \frac{D_i(n)^p}{n^{pl/(l+\beta)}} \leq C \cdot c(n,k)^{-1}{D_i(n)+p-1\choose p} 
\end{align*}
for all $i=1, \ldots, N(n)$ and thus apply Pratt's lemma (\cite{Pratt60}) to finally arrive at \eqref{Eq:Normconvlp} and thereby conclude 1) (b).

\item[3)] For \(p\in (1,\infty)\) the \(p\)-norm on \(\R^d\) among all vectors of equal \(\ell_1\)-norm is minimized if each component takes the same value. Let \(c\) be the cumulated starting weight of the initial \(N(0)\)  vertices. Then
\begin{align*}
\Bigl\lVert\frac{D(n)}{n^{l/(l+\beta)}}\Bigr\rVert_p^p=\frac{1}{n^{p\cdot l/(l+\beta)}}\sum_{i=1}^{\mathclap{N(n)}}D_i^p(n)\geq \frac{(N(0)+n)\cdot (\frac{c+n(l+\beta)}{N(0)+n})^p}{n^{p\cdot l/(l+\beta)}}\sim C\cdot n^{1-\frac{p\cdot l}{l+\beta}}
\end{align*}
for some \(C>0\) and since by assumption \(p<\frac{l+\beta}{l}\) the right-hand side diverges to \(\infty\). Being a.s.\ unbounded, the sequence $D(n)/n^{l/(l+\beta)}$ thus fails to converge.
\end{enumerate}
\end{proof}



\fund 
\noindent There are no funding bodies to thank relating to the creation of this article.

\competing 
\noindent There were no competing interests to declare which arose during the preparation or publication process of this article.

%
%
%
%

\bibliography{bib}

\end{document}